\documentclass[dvipdfmx-if-dvi,autodetect-engine,ja=standard]{amsart}
\usepackage{amsmath,amsthm,amssymb}

\usepackage{mathtools}
\mathtoolsset{showonlyrefs=true} 

\numberwithin{equation}{section} 

\usepackage{amsfonts}
\usepackage{mathrsfs}
\usepackage{bbm}

\usepackage[setpagesize=false,colorlinks=true,linkcolor=blue,citecolor=blue]{hyperref}

\newtheorem{definition}{Definition}[section]
\newtheorem{proposition}[definition]{Proposition}
\newtheorem{theorem}[definition]{Theorem}
\newtheorem{lemma}[definition]{Lemma}

\newtheorem{remark}[definition]{Remark}

\DeclareMathOperator{\supp}{supp}
\DeclareMathOperator*{\esssup}{ess\,sup}

\pagebreak

\title[Nonlinear damped wave equation]{
Existence of solutions to
the semilinear damped wave equation with
non-$L^2$ slowly decaying data :
polynomial nonlinearity case
}

\author[M. Ikeda]{Masahiro Ikeda${}^{1}$}
\thanks{${}^{1}$Graduate School of Information Science and Technology,
The University of Osaka,
1-5 Yamadaoka, 565-0871, Suita, Osaka, Japan%
}
\email{ikeda@ist.osaka-u.ac.jp/masahiro.ikeda@keio.jp/masahiro.ikeda@riken.jp}

\author[T. Inui]{Takahisa Inui${}^2$}
\thanks{
${}^2$Department of Mathematics, Graduate School of Science,
The University of Osaka,
Toyonaka, Osaka 560-0043, Japan%
}
\email{inui@math.sci.osaka-u.ac.jp}

\author[Y. Wakasugi]{Yuta Wakasugi${}^3$}
\thanks{${}^3$
Laboratory of Mathematics,
Graduate School of Advanced Science and Engineering,
Hiroshima University,
Higashi-Hiroshima, 739-8527, Japan%
}
\email{wakasugi@hiroshima-u.ac.jp}

\begin{document}

\maketitle

\begin{abstract}
We study the Cauchy problem of
the semilinear damped wave equation with polynomial nonlinearity,
and establish
the local and global existence of the solution
for slowly decaying initial data
not belonging to $L^2(\mathbb{R}^n)$
in general.
Our approach is based on
the $L^p$-$L^q$ estimates of
linear solutions
and the fractional Leibniz rule
in suitable homogeneous Besov spaces.
\end{abstract}

\section{Introduction}
We consider the Cauchy problem of
the semilinear damped wave equation
\begin{align}\label{eq:ndw}
    \left\{
    \begin{alignedat}{2}
    &\partial_t^2 u - \Delta u + \partial_t u = \mathcal{N}(u),
    &\quad&t>0, x \in \mathbb{R}^n,\\
    &u(0,x) = u_0(x), \ \partial_t u(0,x) = u_1(x),
    &\quad&x\in \mathbb{R}^n.
    \end{alignedat}
    \right.
\end{align}
Here,
$n \ge 1$,
$u = u(t,x)$
is a real-valued unknown function,
$\mathcal{N}(u)$
denotes the power nonlinearity,
and $(u_0,u_1)$ is a given initial data.

The aim of this paper is to
establish the local and global existence of 
solutions for slowly decaying initial data
not belonging to
$L^2(\mathbb{R}^n)$
in general.

After the pioneering work by
Matsumura \cite{Ma76},
the global existence and
asymptotic behavior of solutions to
the semilinear damped wave equation
\eqref{eq:ndw} have been intensively studied.
In particular, the so-called critical exponent, which is the threshold order of nonlinearity for the existence and
nonexistence of global solutions with small initial data, has been investigated.

When
$\mathcal{N}(u) = |u|^p$ with $p>1$
and the initial data
$(u_0,u_1)$ belong to
the energy space
$H^1(\mathbb{R}^n) \times L^2(\mathbb{R}^n)$
with compact support,
Todorova and Yordanov \cite{ToYo01}
and Zhang \cite{Zh01}
determined that the critical exponent
of \eqref{eq:ndw} is given by
the so-called Fujita critical exponent
$p_F = 1+2/n$
named after \cite{Fu66}.
More precisely,
if $p > p_F$
(and $p \le \frac{n}{n-2}$ if $n \ge 3$),
then the unique global solution exists for
small initial data having compact support;
if
$1 < p \le p_F$,
then there is no global solution,
provided that
$u_0+u_1 \in L^1(\mathbb{R}^n)$
and it has positive integral value.

We expect that
the compactness assumption on the support of
the initial data for the existence part
can be relaxed.
More precisely, about the decay property of
the initial data,
it seems that
$L^1(\mathbb{R}^n)$
(or slightly smaller space)
is sufficient
to obtain the global existence
in the case $p > p_F$.
In fact, Ikehata and Ohta \cite{IkOh02}
proved the small data global existence
for the initial data
$u_0 \in H^1(\mathbb{R}^n) \cap L^1(\mathbb{R}^n)$
and
$u_1 \in L^2(\mathbb{R}^n) \cap L^1(\mathbb{R}^n)$
when
$p > p_F$
and $n=1, 2$.
Ono \cite{On03} extended this result to
$n \le 3$.
Similar results
including the asymptotic behavior
of the global solution when
$n=2$ and $n=3$
were given by
Hosono and Ogawa \cite{HoOg04}
and Nishihara \cite{Ni03},
respectively.
The cases $n=4,5$ were done by
Narazaki \cite{Na04}.

The general higher dimensional cases
were settled by
Hayashi, Kaikina, and Naumkin \cite{HaKaNa04DIE}
and Ikehata and Tanizawa \cite{IkTa05}.
The initial data are taken from suitable weighted Sobolev spaces
embedded in $L^1(\mathbb{R}^n)$
in \cite{HaKaNa04DIE}
and 
the energy class
$H^1(\mathbb{R}^n) \times L^2(\mathbb{R}^n)$
with an exponentially increasing weight function
in \cite{IkTa05}, respectively.

From the above studies,
when the initial data belong to
$L^1(\mathbb{R}^n)$,
the critical exponent problem
seems to be almost completed.
Then, the question that what happens
if the initial data does not belong to
$L^1(\mathbb{R}^n)$
naturally arises.

For this problem,
when
$\mathcal{N}(u) = |u|^{p-1}u$,
Nakao and Ono \cite{NaOn93}
proved the small data global existence
of the solution
for the initial data
$(u_0, u_1)\in H^1(\mathbb{R}^n) \times L^2(\mathbb{R}^n)$
if $p \ge 1+4/n$.
In \cite{IkOh02}, the initial data
$u_0 \in H^1(\mathbb{R}^n) \cap L^r(\mathbb{R}^n)$
and
$u_1 \in L^2(\mathbb{R}^n) \cap L^r(\mathbb{R}^n)$
were treated, where
$r \in [1,2]$ when $n=1,2$
and
$r \in [\frac{\sqrt{n^2+16n}-n}{4}, \min \{ 2, \frac{n}{n-2}\} )$
when $3\le n \le 6$,
and the small data global existence was obtained
for $p > 1+2r/n$
(see also \cite{IkTaWa} for
an improvement including the case
$p = 1+2r/n$).
Moreover, in \cite{IkOh02}, the finite time blow-up of
solutions was proved
when $1<p<1+2r/n$.

In the case
$\mathcal{N}(u) = |u|^p$,
by \cite{IkInOkWa19},
the critical exponent is determined as
$p = 1+2r/n$
for the initial data
$(u_0, u_1) \in (H^s(\mathbb{R}^n) \cap H_r^{\beta}(\mathbb{R}^n)) \times
(H^{s-1}(\mathbb{R}^n) \cap L^r(\mathbb{R}^n)) $ with
$r \in (1,2]$ satisfying
$r \ge \frac{2(n-1)}{n+1}$,
$\beta = (n-1)|\frac{1}{2}-\frac{1}{r}|$,
and suitable $s \ge 0$,
where
$H^{\beta}_r(\mathbb{R}^n)$ denotes the Sobolev space
with integrability exponent $r$ and smoothness exponent $\beta$.

The above studies suggest that,
for $r \in (2,\infty)$,
the critical exponent will be still given by
$p_F(n,r) = 1+2r/n$
when the decay of initial data at the spatial infinity is 
$L^r(\mathbb{R}^n)$.
Moreover, for $r >1$, the critical case
$p = p_F(n,r)$
belongs to the existence case,
while we have the blow-up result
for $p= p_F(n,1)$
when
$r = 1$.
This conjecture corresponds to
the well-known result by
Weissler \cite{We81}
for the semilinear heat equation.

However, most of the previous results
considered only the case where
$r \in [1,2]$
and the initial data belong to $L^2(\mathbb{R}^n)$,
except for the result by
Narazaki and Nishihara \cite{NaNi08}.
In \cite{NaNi08},
for $n=1,2,3$,
they
considered the initial data in
a weighted $L^{\infty}$-space
and proved the small data global existence.
However, this approach does not seem to work
in higher space dimensions
due to the derivative loss.
In this paper,
we aim to solve the above conjecture
for the initial data
not belonging to
$L^2(\mathbb{R}^n)$
in general
including the higher dimensional cases.
In particular, we treat the case where
the nonlinearity is a polynomial.
This case enables us to avoid complicated
nonlinear estimates and
is suitable for constructing the framework
of our approach.
The general power-type nonlinearity will be treated in a forthcoming paper.

We also refer the reader to
\cite{Ik04JDE, Ik05JMAA1, On06, So19, IkTaWa}
for the study of the case of exterior domains.
For problems in exterior domains,
the case where the initial data does not belong to $L^2$ is completely open.


To state our results,
we set the following notations and assumptions.
Let
$\mathcal{D}(t)$
denote the solution operator of the linear problem,
that is,
\begin{align}\label{eq:Dt}
    \mathcal{D}(t):=e^{-\frac{t}{2}} \mathcal{F}^{-1} L(t, \xi) \mathcal{F},
\end{align}
where
\begin{align}\label{eq:L}
    L(t, \xi)
    :=
    \begin{dcases}
    \frac{\sinh \left(t \sqrt{1 / 4-|\xi|^{2}}\right)}{\sqrt{1 / 4-|\xi|^{2}}} &(|\xi|<1 / 2), \\
    \frac{\sin \left(t \sqrt{|\xi|^{2}-1 / 4}\right)}{\sqrt{|\xi|^{2}-1 / 4}} & (|\xi|>1 / 2).
    \end{dcases}
\end{align}
Here,
$\mathcal{F}$ and $\mathcal{F}^{-1}$
are the spatial Fourier transform and its inverse,
respectively.

As mentioned above, in what follows,
we assume that the nonlinearity is given by the polynomial:
\begin{equation}\label{assum:N}
    \mathcal{N}(u) = u^p,
    \quad
    \text{where $p$ is an integer greater than $1$.}
\end{equation}
We introduce the definition of solutions
used in this paper.
Let
$r \in (2,\infty)$
and 
$T \in (0,\infty]$.
We say that a function
$u \in C ([0, T) ; \dot{\mathcal{B}}^0_{r,2}(\mathbb{R}^n))$
is a solution of \eqref{eq:ndw} if
\begin{align}\label{eq:def:sol}
    u(t)
    &=
    \mathcal{D}(t) (u_0+u_1) + \partial_t \mathcal{D}(t) u_0
    + \int_0^t \mathcal{D}(t-\tau) \mathcal{N}(u) \,d\tau
\end{align}
holds in
$C ([0, T) ; \dot{\mathcal{B}}^0_{r,2}(\mathbb{R}^n))$.
The definition of the homogeneous Besov spaces $\dot{\mathcal{B}}^s_{p,q}(\mathbb{R}^n)$
will be given in the end of this section.

Our main results are
the local and global existence of solution.

\begin{theorem}\label{thm:lwp}
Let $n \ge 1$ and 
assume that the nonlinearity
$\mathcal{N}(u)$
is given by \eqref{assum:N}.
Let
$r \in (2, \infty)$
and define
$\beta := (n-1) \left(\frac{1}{2} - \frac{1}{r} \right)$.
Assume that there exists
$s \in \mathbb{R}$
satisfying the following conditions:
\begin{itemize}
    \item[(i)] $s > n \left(\frac{1}{2}-\frac{1}{r}\right)$.
    \item[(ii)]
    $p < 1+ \frac{n}{n-2s}$ if $2s<n$,
    and
    \begin{align}
    \min \left\{\frac{r}{2}, 1+\frac{r}{2s}-\frac{1}{s} \right\}
    \leq
    p
    \begin{dcases}
        <\infty & (2s \geq n), \\
    \le
    1+\frac{2}{n-2s}
    & (2s<n).
    \end{dcases}
\end{align}
	\item[(iii)] Either 
	$(2n-1)(\frac{1}{2}-\frac{1}{r}) \le s$ \ or, \  
	$\beta \le 1$ and $p \begin{dcases}
		< \infty &(2s \ge n),\\
		\le \frac{2n}{n-2s} \left( \frac{1}{r} + \frac{1-\beta}{n} \right) &(2s<n). \end{dcases} $
\end{itemize}
Let the initial data satisfy
\begin{align}
    &u_{0} \in
    \dot{\mathcal{B}}^{s}_{2,2} \left(\mathbb{R}^{n}\right)\cap
    \dot{\mathcal{B}}^{0}_{r,2} \left(\mathbb{R}^{n}\right),
    \quad
    \chi_{>1}(\nabla) u_0 \in \dot{\mathcal{B}}_{r,2}^{\beta} \left(\mathbb{R}^{n}\right),
    \\
    &u_{1} \in
    \dot{\mathcal{B}}_{2,2}^{s-1} \left(\mathbb{R}^{n}\right) \cap
    \dot{\mathcal{B}}^{0}_{r,2} \left(\mathbb{R}^{n}\right),
    \quad
    \chi_{>1}(\nabla) u_1 \in \dot{\mathcal{B}}_{r,2}^{\beta-1} \left(\mathbb{R}^{n}\right),
\end{align}
where $\chi_{>1}$ is the cut-off function defined in \eqref{eq:chi} later.
Then, there exist $T>0$ and a unique solution
\begin{align}
    u \in C ([0, T) ;
    \dot{\mathcal{B}}^{s}_{2,2} \left(\mathbb{R}^{n}\right) \cap
    \dot{\mathcal{B}}^{0}_{r,2} \left(\mathbb{R}^{n}\right) )
\end{align}
to the problem \eqref{eq:ndw}.
\end{theorem}

We give some remarks on assumptions in Theorem \ref{thm:lwp}.
\begin{remark}\label{rem:lwp}
\textup{(i)}
The assumption
$p< 1+\frac{n}{n-2s}$
if $2s<n$
implies
$p < \frac{2 n}{n-2 s}$,
which will be frequently used
throughout this paper.

\noindent
\textup{(ii)}
Since $p$ is an integer greater than or equal to $2$,
some restrictions on the parameters $n, r, s$ appear.
For example, we should have
$2 \le \frac{2n}{n-2s}\left(\frac{1}{r}+\frac{1-\beta}{n}\right)$ if $2s < n$.
This and $\beta \ge 0$ imply that $s$ must satisfy $\frac{n}{2}-1 \le 2s$.

On the other hand, there are no restrictions on
the parameter $s$ from above, since the nonlinearity is smooth.
Indeed, the assumptions \textup{(i)--(iii)} in Theorem \ref{thm:lwp} are fulfilled
for given $r \in (2, \infty)$ and integer $p \ge 2$,
provided that $s$ is sufficiently large.

\noindent
\textup{(iii)}
For the assumption of the initial data, basically,
$\dot{\mathcal{B}}^{s}_{2,2} \left(\mathbb{R}^{n}\right)$
and
$\dot{\mathcal{B}}^{0}_{r,2} \left(\mathbb{R}^{n}\right)$
represent the smoothness and decay property of the data, respectively.
We also note that, for example,
$u_0 = u_1 = \langle x \rangle^{-n/r-\varepsilon}$
with sufficiently small $\varepsilon>0$
satisfy the assumptions in Theorem \ref{thm:lwp} for suitablly large $s$,
and do not belong to $L^2(\mathbb{R}^n)$.
\end{remark}

\begin{theorem}\label{thm:gwp}
In addition to the assumptions of Theorem \ref{thm:lwp},
we further assume
\begin{align}\label{eq:critical}
    p \ge 1 + \frac{2r}{n}.
\end{align}
Then, there exists $\varepsilon_0 > 0$
such that if the initial data satisfy
\begin{align}\label{eq:varepsilon0}
    \quad
    \| u_0 \|_{
    \dot{B}^{s}_{2,2} \cap \dot{B}_{r,2}^0
    }
    + \| \chi_{>1}(\nabla) u_0 \|_{\dot{B}_{r,2}^{\beta}}
    +
    \| u_1 \|_{\dot{B}^{s-1}_{2,2}\cap \dot{B}^{0}_{r,2}}
    + \| \chi_{>1}(\nabla) u_1 \|_{\dot{B}_{r,2}^{\beta-1}}
    \le
    \varepsilon_0,
\end{align}
then the problem \eqref{eq:ndw} admits a
unique global solution
\begin{align}
    u \in
    C ([0, \infty) ; \dot{\mathcal{B}}^{s}_{2,2} \left(\mathbb{R}^{n}\right)
    \cap \dot{\mathcal{B}}^{0}_{r,2} \left(\mathbb{R}^{n}\right)
    ).
\end{align}
\end{theorem}

We explain the strategies for the proofs of above theorems.
Since we assume that the initial data lie in $L^r(\mathbb{R}^n)$
and not in $L^2(\mathbb{R}^n)$
in general, we need to use homogeneous
spaces such as
$\dot{\mathcal{B}}^s_{2,2}(\mathbb{R}^n)$.
Then, in view of the linear and nonlinear estimates
(see Proposition \ref{prop:lplq} and Lemma \ref{lem:nonlin}),
it is natural to use
$\dot{\mathcal{B}}^0_{r,2}(\mathbb{R}^n)$
rather than
$L^r(\mathbb{R}^n)$.
However, 
in a standard definition of the homogeneous space
(see e.g., \cite[Ch. 5, Definition 1]{Tr}),
an element of it is
not a function but an equivalent class,
for which the composition with nonlinear function does not make sense.
To avoid handling with an equivalent class,
we employ the definition of homogeneous Besov space by
Bahouri, Chemin and Danchin \cite{BaChDa}.

Since the nonlinearity is polynomial,
we can estimate the nonlinear term by the fractional Leibniz rule only
(see Lemma \ref{lem:nonlin}).
For general power type nonlinearities, besides the fractional Leibniz rule,
the fractional chain rule will be also needed.

This paper is organized as follows.
In the next section, we prepare
the $L^p$-$L^q$ estimates of
the solution operator
and some basic tools for the nonlinear estimate.
In section 3, we give a proof of
Theorems \ref{thm:lwp} and \ref{thm:gwp}.
Finally, in Appendix,
we present proofs of some auxiliary results
used in the proof of Theorem \ref{thm:lwp}.

\subsection{Notations}
Here, we explain the notations used
throughout this paper.
The symbol
$X \lesssim Y$
stands for
$X \le C Y$
with some constant $C>0$.
$X\sim Y$
means that both
$X \lesssim Y$ and $Y\lesssim X$
hold.

For a function
$f : \mathbb{R}^n \to \mathbb{C}$,
the Fourier and the inverse Fourier transforms are defined by
\begin{align}
    \mathcal{F}[f](\xi)
    &= \hat{f}(\xi)= (2 \pi)^{-n / 2} \int_{\mathbb{R}^{n}} e^{-i x \cdot \xi} f(x) \,d x, \\
    \mathcal{F}^{-1}[f](x)
    &= (2 \pi)^{-n / 2} \int_{\mathbb{R}^{n}} e^{i x \cdot \xi} f(\xi) \,d \xi.
\end{align}
For a function
$m : \mathbb{R}^n \to \mathbb{C}$,
we define the Fourier multiplier
$m(\nabla)$
by
\begin{align}
    m(\nabla) f(x)=\mathcal{F}^{-1}[m(\xi) \hat{f}(\xi)](x).
\end{align}

Let
$\chi \in C_0^{\infty}(\mathbb{R}^n)$
be a cut-off function such that
\begin{align}
    0 \le \chi(\xi) \le 1 \quad (\xi \in \mathbb{R}^n),\qquad
    \chi(\xi) = \begin{dcases}
    1 &(|\xi| \le 1),\\
    0 &( |\xi| \ge \frac{25}{24} ).
    \end{dcases}
\end{align}
For 
$a > 0$,
we define
\begin{align}
\label{eq:chi}
    \chi_{\le a} (\xi)
    := \chi \left( \frac{\xi}{a} \right),
    \quad
    \chi_{> a} (\xi)
    := 1 - \chi_{\le a} (\xi),
    \quad
    \chi_a(\xi) := \chi_{\le a}(\xi) - \chi_{\le \frac{a}{2}}(\xi).
\end{align}
For $j \in \mathbb{Z}$,
we also set
\begin{align}
    \widehat{\Delta_{\leq j} f}(\xi) &:=\chi_{\leq 2^{j}}(\xi) \widehat{f}(\xi), \\
    \widehat{\Delta_{>j} f}(\xi) &:=\chi_{>2^{j}}(\xi) \widehat{f}(\xi), \\
    \widehat{\Delta_{j} f}(\xi) &:=\chi_{2^{j}}(\xi) \widehat{f}(\xi).
\end{align}
Moreover, we assume that
$\{ \chi_{2^j} \}_{j\in \mathbb{Z}}$
is a dyadic decomposition of unity:
\begin{align}
    \chi_{\leq 1} + \sum_{j=1}^{\infty} \chi_{2^j} = 1 \ \text { on } \ 
    \mathbb{R}^n \quad
    \text { and }
    \quad
    \sum_{j \in \mathbb{Z}} \chi_{2^j}=1
    \ \text { on } \ 
    \mathbb{R}^n \setminus \{0\}.
\end{align}
We sometimes use
\begin{align}
    \widetilde{\Delta}_j f
    = \Delta_{j-1} f + \Delta_j f + \Delta_{j+1} f.
\end{align}
Then, by the support property of
$\chi$,
we have
$\Delta_j f = \widetilde{\Delta}_j \Delta_j f$.

Next, we define function spaces used in this paper.
First,
$\mathcal{S}(\mathbb{R}^n)$
and
$\mathcal{S}'(\mathbb{R}^n)$
are
the Schwartz space and its dual, respectively.

Next, following \cite{BaChDa} and \cite{HiWa19}, we introduce
\begin{align}
    \mathcal{S}_{\ast} (\mathbb{R}^n)
    &:= \left\{ f \in \mathcal{S}(\mathbb{R}^n) ; 0 \notin \supp \hat{f} \right\},\\
    \mathcal{S}'_h (\mathbb{R}^n)
    &:= \left\{ f \in \mathcal{S}'(\mathbb{R}^n) ; \lim_{a\to 0} \| \chi_{\le a}(\nabla) f \|_{L^{\infty}} = 0 \right\},\\
    \mathcal{S}'_0 (\mathbb{R}^n)
    &:= \left\{ f \in \mathcal{S}'(\mathbb{R}^n) ; \lim_{a\to 0} \chi_{\le a}(\nabla) f = 0 \ \text{in} \ \mathcal{S}'(\mathbb{R}^n) \right\}.
\end{align}
It is obvious that
$\mathcal{S}'_h(\mathbb{R}^n) \subset
\mathcal{S}'_0(\mathbb{R}^n)$.
Moreover, we have, for any
$f \in \mathcal{S}'(\mathbb{R}^n)$,
\begin{align}
    \lim_{a \to \infty} \Delta_{\le a}f
    = \Delta_{\le 0}f + \sum_{j=1}^{\infty} \Delta_j f  = f
    \quad \text{in} \ \mathcal{S}'(\mathbb{R}^n),
\end{align}
and for any $f \in \mathcal{S}'_0(\mathbb{R}^n)$,
\begin{align}
    \sum_{j=-\infty}^{\infty} \Delta_j f  = f
    \quad \text{in} \ \mathcal{S}'(\mathbb{R}^n).
\end{align}

For
$1 \le p \le \infty$,
$L^p = L^p(\mathbb{R}^n)$
denotes the usual Lebesgue space
equipped with the norm
\begin{align}
    \| f \|_{L^p}
    := \begin{dcases}
    \left( \int_{\mathbb{R}^n}
    |f(x)|^p \,dx \right)^{1/p}
    &(1\le p < \infty),\\
    \esssup_{x\in \mathbb{R}^n} |f(x)|
    &(p=\infty).
    \end{dcases}
\end{align}
For
$s \in \mathbb{R}$
and $1 \le p \le \infty$,
the Sobolev space
$H_p^s = H_p^s(\mathbb{R}^n)$
is defined by
\begin{align}
    H_p^s(\mathbb{R}^n)
    &:=
    \left\{ f \in \mathcal{S}'(\mathbb{R}^n) ; \ 
    \| f \|_{H_p^s} < \infty \right\},\\
    \| f \|_{H_p^s}
    &:=
    \| \langle \nabla \rangle^s f \|_{L^p(\mathbb{R}^n)},
\end{align}
where
$\langle \xi \rangle := \sqrt{1+|\xi|^2}$.
When $p=2$, we will write $H^s_2(\mathbb{R}^n) = H^s (\mathbb{R}^n)$.

Next, for
$s \in \mathbb{R}$,
$1 \le p, q \le \infty$,
and
$f \in \mathcal{S}'(\mathbb{R}^n)$,
we define the homogeneous Besov seminorm of $f$ by
\[
	\| f \|_{\dot{B}^s_{p,q}}
	:=
	\left\|\left\{2^{j s}\left\|\Delta_{j} f\right\|_{L^{p}}\right\}_{j \in \mathbb{Z}} \right\|_{l^{q}},
\]
where $l^q$-norm is given by
\begin{align}
    \left\| \{ a_j \}_{j\in \mathbb{Z}} \right\|_{l^q}
    := \begin{dcases}
    \left( \sum_{j=-\infty}^{\infty} |a_j|^q \right)^{1/q}
    &(1\le q < \infty),\\
    \sup_{j \in \mathbb{Z}} |a_j|
    &(q = \infty).
    \end{dcases}
\end{align}
If $f \in \mathcal{S}'(\mathbb{R}^n)$ satisfies $\| f \|_{\dot{B}^s_{p,q}} = 0$,
then $\supp \hat{f} \subset \{ 0 \}$ holds and $f$ is a polynomial.
In order to make $\| \cdot \|_{\dot{B}^s_{p,q}}$ be a norm,
following \cite{BaChDa},
we replace $\mathcal{S}'(\mathbb{R}^n)$ by $\mathcal{S}'_h(\mathbb{R}^n)$
and define the homogeneous Besov space by
\begin{align}
    \dot{\mathcal{B}}^{s}_{p,q}(\mathbb{R}^n)
    &:=
    \left\{
    f \in \mathcal{S}'_h(\mathbb{R}^n) ;\ 
    \| f \|_{\dot{B}^s_{p,q}} < \infty
    \right\}.
\end{align}

We will sometimes use the inhomogeneous Besov space defined by
\begin{align*}
    B^s_{p,q}(\mathbb{R}^n) &:= 
    \left\{ f \in \mathcal{S}'(\mathbb{R}^n) ;\, \| f \|_{B^s_{p,q}} < \infty \right\},\\
    \| f \|_{B^s_{p,q}} &:=
    \| \Delta_{\le 0} f \|_{L^p}
    + \left\| \left\{ 2^{js} \| \Delta_j f \|_{L^p} \right\}_{j \in \mathbb{N}} \right\|_{l^q},
\end{align*}
where
$\| \{ a_j \}_{j\in \mathbb{N}} \|_{l^q}$
is defined in the same way as above.

For a Banach space $X$ and an interval $I \subset \mathbb{R}$,
we define
$C(I; X)$
as the space of continuous functions from $I$ to $X$.
Moreover, $L^{\infty}(I ; X)$
denotes the space of strongly measurable essentially bounded functions from $I$ to $X$.

\subsection{Basic facts on functions spaces}
Here, we collect some basic properties on the function spaces defined above.
We will use the following facts throughout the paper.

\begin{enumerate}
\item
$\mathcal{S}(\mathbb{R}^n)$ is a separable space (\cite[p.144]{ReSiI}).

\item
For $q \in [1,\infty)$, the embedding $L^q(\mathbb{R}^n) \subset \mathcal{S}'_h(\mathbb{R}^n)$ holds.

\item
For $s\in \mathbb{R}$ and $p, q \in [1,\infty)$,
the embedding $\mathcal{S}_{\ast}(\mathbb{R}^n) \subset \dot{\mathcal{B}}^s_{p,q}(\mathbb{R}^n)$ 
is dense (\cite[Theorem 2.27]{BaChDa}).

\item
If $s \in \mathbb{R}$ and $p,q\in [1,\infty]$ satisfy
$s < \frac{n}{p}$ or $s=\frac{n}{p}, q=1$,
then
$\dot{\mathcal{B}}^s_{p,q}(\mathbb{R}^n)$
is complete.
Moreover, if $s,p,q$ satisfy the same conditions as above, then 
$\dot{\mathcal{B}}^s_{p,q}(\mathbb{R}^n) \cap \dot{\mathcal{B}}^{\tilde{s}}_{\tilde{p},\tilde{q}}(\mathbb{R}^n)$
is complete for any
$\tilde{s}\in \mathbb{R}$ and $\tilde{p}, \tilde{q} \in [1,\infty]$
(\cite[Theorem 2.25]{BaChDa}).

\item
For
$1 \le p_1 \le p_2 \le \infty$,
$1 \le q_1 \le q_2 \le \infty$
and $s \in \mathbb{R}$,
the embedding
$\dot{\mathcal{B}}^s_{p_1, q_1}(\mathbb{R}^n) \subset
\dot{\mathcal{B}}^{s-n(\frac{1}{p_1}-\frac{1}{p_2})}_{p_2, q_2}(\mathbb{R}^n)$
holds
(\cite[Theorem 2.20]{BaChDa}).

\item
For $q \in [1,\infty]$, the space $L^q(\mathbb{R}^n)$
is continuously embedded in the space $\dot{\mathcal{B}}^0_{q,\infty}(\mathbb{R}^n)$
(\cite[Theorem 2.39]{BaChDa}).

\item
For $q \in [2,\infty)$, the embeddings
$\dot{\mathcal{B}}^0_{q,2} (\mathbb{R}^n) \subset L^q(\mathbb{R}^n)$
and
$B^0_{q,2} (\mathbb{R}^n) \subset L^q(\mathbb{R}^n)$
hold.
Moreover, for $q \in (1,2]$, the embeddings
$L^q(\mathbb{R}^n) \subset \dot{\mathcal{B}}^0_{q,2} (\mathbb{R}^n)$
and
$L^q(\mathbb{R}^n) \subset B^0_{q,2} (\mathbb{R}^n)$
hold
(\cite[Theorem 2.40]{BaChDa}, \cite[Theorem 6.4.4]{BeLo}).

\item 
(The Sobolev embedding)
The embedding
$B^s_{p, q} (\mathbb{R}^n) \subset B^{\tilde{s}}_{\tilde{p}, \tilde{q}}(\mathbb{R}^n)$
holds for
$1 \le p \le \tilde{p} \le \infty, 1 \le q \le \tilde{q} \le \infty$ and $s, \tilde{s} \in \mathbb{R}$
satisfying
$s - \frac{n}{p} = \tilde{s} - \frac{n}{\tilde{p}}$.
Moreover, the embedding
$H^s_p(\mathbb{R}^n) \subset H^{\tilde{s}}_{\tilde{p}}(\mathbb{R}^n)$
holds for 
$1 < p \le \tilde{p} < \infty$ and  $s, \tilde{s} \in \mathbb{R}$
satisfying
$s - \frac{n}{p} = \tilde{s} - \frac{n}{\tilde{p}}$
(\cite[Theorem 6.5.1]{BeLo}).

\item
(Duality of $\dot{\mathcal{B}}^s_{p,q}(\mathbb{R}^n)$)
For $p, q \in [1,\infty)$ and $s \in \mathbb{R}$ satisfying
$\frac{n}{p}-n < s < \frac{n}{p}$
or
$s = \frac{n}{p}, q=1$,
we have the isomorphism
$( \dot{\mathcal{B}}^s_{p,q}(\mathbb{R}^n))' \simeq \dot{\mathcal{B}}^{-s}_{p',q'}(\mathbb{R}^n)$,
where $p'=\frac{p}{p-1}$ stands for the H\"{o}lder conjugate of $p$
(\cite[Theorem 2.17]{Sa} and \cite[Remark 2.24, Proposition 2.29]{BaChDa}).

\item (The Pettis theorem) 
Let $X$ be a separable Banach space and let $I \subset \mathbb{R}$ be an interval.
Then, a function $u : I \to X$ is strongly measurable if and only if
$u$ is weakly measurable (\cite[Theorem 1.1.6]{HyNiVeWe}).

\item (Measurability of composite functions)
Let $X, Y$ be Banach spaces and let $I \subset \mathbb{R}$ be an interval.
If $f : I \to X$ is strongly measurable and $\phi : X \to Y$ is continuous,
then $\phi \circ f : I \to Y$ is strongly measurable (\cite[Corollary 1.1.11]{HyNiVeWe}).

\end{enumerate}

\section{Preliminaries}

We first recall the
$L^p$-$L^q$ estimates for the operator $\mathcal{D}(t)$
proved in the previous result
\cite[Theorem 1.1]{IkInOkWa19}.
\begin{proposition}[$L^p$-$L^q$ estimates] \label{prop:lplq:Leb}
Let
$1 \le q \leq p<\infty$,
$p \neq 1$,
and
$\beta=(n-1)\left|\frac{1}{2}-\frac{1}{p}\right|$. Then, there exists $\delta_{p}>0$ such that
for $t>0$,
\begin{align}
    &\left\|  \mathcal{D}(t) g\right\|_{L^p}
    \lesssim
    \langle t\rangle^{-\frac{n}{2}\left(\frac{1}{q}-\frac{1}{p}\right)}
    \left\| \chi_{\leq 1}(\nabla) g \right\|_{L^q}
    +e^{-\frac{t}{2}} \langle t\rangle^{\delta_{p}}
    \left\| \chi_{>1}(\nabla) g\right\|_{H_p^{\beta-1}}
\end{align}
for any
$g \in \mathcal{S}'(\mathbb{R}^n)$
satisfying
$\chi_{\le 1}(\nabla) g \in
L^q(\mathbb{R}^n)$
and
$\chi_{>1}(\nabla) g \in
H_p^{\beta-1}(\mathbb{R}^n)$.
\end{proposition}

In this paper, we mainly use the following
refined version of
$L^p$-$L^q$ estimates in Besov norms.

\begin{proposition}[$L^{p}$-$L^{q}$ estimates in Besov norms]\label{prop:lplq}
Let
$1 \le q \leq p<\infty$,
$p \neq 1$,
$s_{1} \geq s_{2}$,
and
$\beta=(n-1)\left|\frac{1}{2}-\frac{1}{p}\right|$. Then, there exists $\delta_{p}>0$ such that
for $t>0$,
we have
$\mathcal{D}(t)g \in \dot{\mathcal{B}}^{s_1}_{p,2}(\mathbb{R}^n)$
and
\begin{align}
    &\left\|  \mathcal{D}(t) g\right\|_{\dot{B}^{s_1}_{p,2}}
    \lesssim
    \langle t\rangle^{-\frac{n}{2}\left(\frac{1}{q}-\frac{1}{p}\right)-\frac{s_{1}-s_{2}}{2}}
    \left\| \chi_{\leq 1}(\nabla) g \right\|_{\dot{B}^{s_2}_{q,2}}
    +e^{-\frac{t}{2}} \langle t\rangle^{\delta_{p}}
    \left\| \chi_{>1}(\nabla) g\right\|_{\dot{B}_{p,2}^{s_1+\beta-1}}
\end{align}
for any
$g \in \mathcal{S}'_h(\mathbb{R}^n)$
satisfying
$\chi_{\le 1}(\nabla) g \in
\dot{\mathcal{B}}^{s_2}_{q,2}(\mathbb{R}^n)$
and
$\chi_{>1}(\nabla) g \in
\dot{\mathcal{B}}_{p,2}^{s_1+\beta-1}(\mathbb{R}^n)$,
and we have
$\partial_t \mathcal{D}(t)g \in \dot{\mathcal{B}}^{s_1}_{p,2}(\mathbb{R}^n)$ and
\begin{align}
    \left\| \partial_{t} \mathcal{D}(t) g\right\|_{\dot{B}^{s_1}_{p,2}}
    \lesssim\langle t\rangle^{-\frac{n}{2}\left(\frac{1}{q}-\frac{1}{p}\right)-\frac{s_{1}-s_{2}}{2}-1}
    \left\| \chi_{\leq 1}(\nabla) g\right\|_{\dot{B}^{s_2}_{q,2}}
    + e^{-\frac{t}{2}} \langle t\rangle^{\delta_{p}}
    \left\| \chi_{>1}(\nabla) g \right\|_{\dot{B}_{p,2}^{s_1+\beta}}
\end{align}
for any
$g \in \mathcal{S}'_h(\mathbb{R}^n)$
satisfying
$\chi_{\le 1}(\nabla) g \in
\dot{\mathcal{B}}^{s_2}_{q,2}(\mathbb{R}^n)$
and
$\chi_{>1}(\nabla) g \in
\dot{\mathcal{B}}_{p,2}^{s_1+\beta}(\mathbb{R}^n)$.
\end{proposition}

\begin{remark}\label{rem:lplq}
\textup{(i)}
By the Young inequality,
the term
$\left\| \chi_{\leq 1}(\nabla) g\right\|_{\dot{B}^{s_2}_{q,2}}$
in the above estimates can be also replaced by
$\left\| g\right\|_{\dot{B}^{s_2}_{q,2}}$.

\noindent
\textup{(ii)}
Proposition \ref{prop:lplq} can be proved in almost the same way as
\cite[Theorem 1.1]{IkInOkWa19}.
We will give a proof of Proposition \ref{prop:lplq}
in Appendix A.

\noindent
\textup{(iii)}
The third parameter of the Besov norms is not necessarily $2$, and it can be
replaced by any number in $[1,\infty]$
as long as the exponents on both sides are the same.
In this paper, we only use the above version.
\end{remark}

\begin{definition}\label{def:norm}
Let
$s > 0, r \in (2,\infty)$ and
$T \in (0,\infty]$.
Define
\begin{align}
    X(T) = \left\{
    \phi \in L^{\infty}(0,T; \dot{\mathcal{B}}^0_{r,2}(\mathbb{R}^n)) ;
    \| \phi \|_{X(T)} < \infty
    \right\},
\end{align}
where
\begin{align}
    \|\phi\|_{X(T)}
    &:= \esssup _{0 < t<T}
    \left\{ \langle t\rangle^{\frac{s}{2}-\frac{n}{2}\left(\frac{1}{2}-\frac{1}{r}\right)}
    \left\| \phi(t) \right\|_{\dot{B}^{s}_{2,2}}
    + \|\phi(t)\|_{\dot{B}^{0}_{r,2}} \right\}.
\end{align}
Let
\begin{align}
    \eta := \frac{s-1}{2}+\frac{n}{2}\left(\frac{p}{r}-\frac{1}{2}\right),\quad
    \sigma_{1}:=\max \left\{1, \frac{r}{p}\right\} + \varepsilon,
    \quad
    \sigma_{2}:=
    \begin{dcases}
    r & (2 s \geq n), \\
    \min \left\{r, \frac{2 n}{p(n-2s)}\right\} & (2 s<n),
    \end{dcases}
\end{align}
and
$\varepsilon>0$
is a sufficiently small constant so that
$\sigma_1 < \sigma_2$
holds
(see Remark \ref{rem:sigma1:sigma2} (i) below).
We also define
\begin{align}
    Y(T) = \left\{ \psi \in L^{\infty}(0,T;
     \dot{\mathcal{B}}^0_{\sigma_1,2}(\mathbb{R}^n))
    ; \| \psi \|_{Y(T)} < \infty \right\},
\end{align}
where
\begin{align}
    \| \psi \|_{Y(T)}
    &:=
    \begin{dcases}
    \esssup _{0 < t<T} \left\{
    \langle t\rangle^{\eta} \left\| \chi_{>1}(\nabla) \psi(t)\right\|_{\dot{B}^{s-1}_{2,2}}
    +\max_{\gamma \in \left[\sigma_{1}, \sigma_{2}\right]}
    \langle t\rangle^{\frac{n}{2}\left(\frac{p}{r}-\frac{1}{\gamma}\right)} \|\psi(t)\|_{\dot{B}^0_{\gamma,2}} \right\}
    &(0 < s \le 1),\\
    \esssup _{0 < t<T}
    \left\{
    \langle t\rangle^{\eta} \left\| \psi(t) \right\|_{\dot{B}^{s-1}_{2,2}}
    +\max_{\gamma \in\left[\sigma_1, \sigma_{2}\right]}
    \langle t\rangle^{\frac{n}{2}\left(\frac{p}{r}-\frac{1}{\gamma}\right)}\|\psi(t)\|_{\dot{B}^0_{\gamma,2}} \right\}
    &(s>1).
    \end{dcases}
\end{align}
\end{definition}

\begin{remark}\label{rem:sigma1:sigma2}
{\rm (i)}
The assumptions of Theorem \ref{thm:lwp} imply 
$r < \frac{2n}{n-2s}$,
$p < 1+ \frac{n}{n-2s}$,
and $s > 0$,
which ensure
$\sigma_1 < \sigma_2$,
provided that
$\varepsilon$
is sufficiently small.

\noindent
{\em (ii)}
Adding
$\varepsilon$
in the definition of $\sigma_1$ above is needed for
the proof of Lemmas \ref{lem:Duha} and \ref{lem:nonlin}.
\end{remark}

In the following,
we introduce some estimates
which will be used throughout
this paper.

\begin{lemma}\label{lem:int}
Let
$T \in (0,\infty]$,
$r \in (2,\infty)$, and $s > n \left( \frac{1}{2} - \frac{1}{r} \right)$.
Assume
$q \in [1,\infty]$,
$\alpha \in \mathbb{R}$,
and
$\theta \in [0,1]$
satisfy
\begin{align}
    \frac{n}{q}-\alpha=(1-\theta) \frac{n}{r}+\theta\left(\frac{n}{2}-s\right), \quad \alpha \leq \theta s.
\end{align}
Then,
for any $\phi \in X(T)$, we have
$\phi(t) \in \dot{\mathcal{B}}^{\alpha}_{q,2}(\mathbb{R}^n)$ for almost every
$t \in (0,T)$
and
\begin{align}
    \esssup_{0<t<T} \,\langle t\rangle^{\frac{n}{2}\left(\frac{1}{r}-\frac{1}{q}\right)+\frac{\alpha}{2}}\|\phi(t)\|_{ \dot{B}^{\alpha}_{q,2}}
    \lesssim
    \|\phi\|_{X(T)},
\end{align}
where the implicit constant is independent of $\phi$.
\end{lemma}
This lemma immediately follows from
\cite[Theorem 2.1]{HaMoOzWa11}.

In the estimates of nonlinearity,
we employ the following
fractional Leibniz rule.

\begin{lemma}[Fractional Leibniz rule]\label{lem:frac:Leib}
Let
$\alpha>0$,
$1 \leq r, \sigma, p_{1}, p_{2}, q_{1}, q_{2} \leq \infty$
and
$ \frac{1}{r}
=\frac{1}{p_{1}}+\frac{1}{q_{1}}
=\frac{1}{p_{2}}+\frac{1}{q_{2}}$.
In addition, we assume
$r, q_1, q_2 < \infty$.
Then,
for any
$f \in \dot{\mathcal{B}}_{p_{1}, \sigma}^{\alpha}(\mathbb{R}^n) \cap L^{q_{2}}(\mathbb{R}^n)$
and
$g \in \dot{\mathcal{B}}_{p_{2}, \sigma}^{\alpha}(\mathbb{R}^n) \cap L^{q_{1}}(\mathbb{R}^n)$,
we have the estimate
\begin{align}
    \|f g\|_{\dot{B}_{r, \sigma}^{\alpha}} 
    \lesssim
    \|f\|_{\dot{B}_{p_{1}, \sigma}^{\alpha}}\|g\|_{L^{q_{1}}}
    + \|g\|_{\dot{B}_{p_{2}, \sigma}^{\alpha}}
    \|f\|_{L^{q_{2}}},
\end{align}
where the implicit constant is independent of $f, g$.
\end{lemma}

\begin{remark}\label{rem:frac:Leib}
\textup{(i)}
Under the assumptions of Lemma \ref{lem:frac:Leib}, we can see that
$fg \in \mathcal{S}'(\mathbb{R}^n)$ (see Appendix \ref{app:FraLei}).
Thus, the left-hand side of the above estimate makes sense as a seminorm of $fg \in \mathcal{S}'(\mathbb{R}^n)$.

\noindent
\textup{(ii)}
When $r < \infty$ and $(q_1, q_2) \neq (\infty, \infty)$, we can prove
$f g \in \mathcal{S}_h'(\mathbb{R}^n)$.
This will be shown in Appendix \ref{app:FraLei}.
We remark that the condition $(q_1, q_2) \neq (\infty, \infty)$ is necessary,
since $f = e^{ix\cdot \xi_0}$, $g = e^{-ix\cdot \xi_0}$ with $\xi_0 \neq 0$
satisfy $fg=1 \notin \mathcal{S}_h'(\mathbb{R}^n)$.

\noindent
\textup{(iii)}
The proof of Lemma \ref{lem:frac:Leib} is based on the paraproduct decomposition.
Under the assumptions of Lemma \ref{lem:frac:Leib} with $r, q_1, q_2 < \infty$,
we can prove the paraproduct decomposition in $\mathcal{S}'(\mathbb{R}^n)$.
We will also give its proof of paraproduct decomposition in Appendix \ref{app:FraLei}.
\end{remark}

The proof of the estimate above is the same as
\cite[Proposition 2.2]{KaKoSh19}
and we omit the detail.

Finally, we show the measurability of the function
$\mathcal{D}(t-\tau) \psi (\tau)$
for the Bochner integral of the Duhamel term to make sense.
\begin{lemma}\label{lem:meas:psi}
For $\psi \in Y(T)$,
$r \in (2,\infty)$,
and $t > 0$,
the function
$\tau \mapsto \mathcal{D}(t-\tau) \psi(\tau)$
is strongly measurable from
$[0,t] \to \dot{\mathcal{B}}^0_{r,2}(\mathbb{R}^n) \cap \dot{\mathcal{B}}^s_{2,2}(\mathbb{R}^n)$.
\end{lemma}
Moreover, to apply the above lemma for the nonlinearity $\mathcal{N}(u)=u^p$,
we need the following.
\begin{lemma}\label{lem:meas:N}
Under the assumption \eqref{assum:N},
for $u \in X(T)$,
the function
$u^p : [0,T] \to \dot{\mathcal{B}}^{0}_{\sigma_1,2}(\mathbb{R}^n)$
is strongly measurable.
\end{lemma}
The proofs of the above two lemmas are given in Appendix \ref{app:C}.

\section{Local and global existence}
In this section, we prove Theorems \ref{thm:lwp} and \ref{thm:gwp}.
We start with the following lemma.

\begin{lemma}\label{lem:Duha}
Under the assumptions of Theorem \ref{thm:lwp},
for $\psi \in Y(T)$,
we have
\begin{align}
    \left\|\int_{0}^{t} \mathcal{D}(t-\tau) \psi(\tau) d\tau \right\|_{X(T)} \lesssim \int_{0}^{T}\|\psi\|_{Y(\tau)} d \tau
\end{align}
for $0<T<1$, and
\begin{align}
    \left\|\int_{0}^{t} \mathcal{D}(t-\tau) \psi(\tau) d\tau \right\|_{X(T)} \lesssim
    \| \psi \|_{Y(T)}
    \begin{dcases}
    1&(p\ge 1+\frac{2r}{n}),\\
    \langle T \rangle^{1-\frac{n}{2r}(p-1)}
    &(1<p<1+\frac{2r}{n})
    \end{dcases}
\end{align}
for any $T>0$,
where the implicit constants are independent of $T$.
\end{lemma}
\begin{proof}
By Lemma \ref{lem:meas:psi},
the Bochner integral
$\int_0^t \mathcal{D}(t-\tau) \psi (\tau) \,d\tau$
in
$\dot{\mathcal{B}}^0_{r,2}(\mathbb{R}^n) \cap \dot{\mathcal{B}}^{s}_{2,2}(\mathbb{R}^n)$
makes sense for every $t \in [0,T]$.
Thus, it remains to prove the above two estimates.

The first assertion can be easily
obtained by slightly modifying the proof of the second one.
Thus, we only prove the second assertion.

\noindent
\textbf{(Step 1)}
First, we estimate
\begin{align}
    \langle t\rangle^{\frac{n}{2}\left(\frac{1}{r}-\frac{1}{2}\right)+\frac{s}{2}}
    \left\| \int_{0}^{t} \mathcal{D}(t-\tau) \psi(\tau) d \tau \right\|_{\dot{B}^s_{2,2}}
    &\lesssim
    \langle t\rangle^{\frac{n}{2}\left(\frac{1}{r}-\frac{1}{2}\right)+\frac{s}{2}} \int_{0}^{\frac{t}{2}}\left\| \mathcal{D}(t-\tau) \psi(\tau)\right\|_{\dot{B}^s_{2,2}} \,d \tau \\
    &\quad +
    \langle t\rangle^{\frac{n}{2}\left(\frac{1}{r}-\frac{1}{2}\right)+\frac{s}{2}} \int_{\frac{t}{2}}^t
        \left\| \mathcal{D}(t-\tau) \psi(\tau)\right\|_{\dot{B}^s_{2,2}} \,d \tau \\
    &=: I_1 + I_2.
\end{align}
For $I_1$,
applying Proposition \ref{prop:lplq} with
$p=2, q=\sigma_1, s_1 = s, s_2 = 0$
(see also Remark \ref{rem:lplq} (i)),
we have
\begin{align}
    I_1
    &\lesssim\langle t\rangle^{\frac{n}{2}\left(\frac{1}{r}-\frac{1}{2}\right)+\frac{s}{2}}
    \int_{0}^{\frac{t}{2}}
    \left(\langle t-\tau \rangle^{-\frac{n}{2}\left(\frac{1}{\sigma_{1}}-\frac{1}{2}\right)-\frac{s}{2}}
    \|\psi(\tau)\|_{\dot{B}^0_{\sigma_1,2}}\right.\\
    &\left.\qquad\qquad\qquad\qquad \qquad 
    + e^{-\frac{t-\tau}{2}}
    \langle t - \tau \rangle^{\delta_2}
    \left\| \chi_{>1}(\nabla) \psi(t)\right\|_{\dot{B}^{s-1}_{2,2}}
    \right) \,d \tau \\
    &\lesssim
    \langle t\rangle^{-\frac{n}{2}\left(\frac{1}{\sigma_{1}}-\frac{1}{r}\right)}
    \int_{0}^{\frac{t}{2}}
    \left(
        \langle\tau\rangle^{-\frac{n}{2}\left(\frac{p}{r}-\frac{1}{\sigma_1}\right)}
        +e^{-\frac{t-\tau}{4}}
        \langle t-\tau \rangle^{\delta_2}
        \langle\tau\rangle^{-\eta}
    \right) \,d \tau\|\psi\|_{Y(T)} \\
    &\lesssim
    \| \psi \|_{Y(T)}
    \langle T\rangle^{\left( 1-\frac{n}{2 r}(p-1) \right)_+},
\end{align}
where we write
$\left( 1-\frac{n}{2 r}(p-1) \right)_+ =
\max\{0, 1-\frac{n}{2 r}(p-1)\}$.
To estimate $I_2$, we first consider the case
$s>1$.
By Proposition \ref{prop:lplq} with
$p=q=2, s_1 = s, s_2 = s-1$,
we deduce
\begin{align}
    I_2
    &\lesssim
    \langle t\rangle^{\frac{n}{2}\left(\frac{1}{r}-\frac{1}{2}\right)+\frac{s}{2}}
    \int_{\frac{t}{2}}^{t} 
    \left(
        \langle t-\tau \rangle^{-\frac{1}{2}}
        \left\| \psi(\tau) \right\|_{\dot{B}^{s-1}_{2,2}}
        + e^{-\frac{t-\tau}{2}} 
        \langle t-\tau \rangle^{\delta_2}
        \left\| \psi(\tau) \right\|_{\dot{B}^{s-1}_{2,2}} 
    \right) \, d \tau \\
    &\lesssim
    \langle t\rangle^{\frac{n}{2}\left(\frac{1}{r}-\frac{1}{2}\right)+\frac{s}{2}}
    \int_{\frac{t}{2}}^{t}
    \langle t-\tau \rangle^{-\frac{1}{2}}
    \langle\tau\rangle^{-\eta} \,d\tau
    \|\psi\|_{Y(T)} \\
    &\lesssim
    \langle t\rangle^{-\frac{n}{2r}\left( p-1 \right)+\frac{1}{2}}
    \int_{\frac{t}{2}}^{t}
    \langle t-\tau \rangle^{-\frac{1}{2}}
    \,d \tau \|\psi\|_{Y(T)} \\
    &\lesssim
    \| \psi \|_{Y(T)}
    \langle T\rangle^{\left( 1-\frac{n}{2 r}(p-1) \right)_+}.
\end{align}
Next, we treat the case $0 < s \le 1$.
We apply
Proposition \ref{prop:lplq}
with
$p=2, s_1 = s, s_2 = 0$,
and
$q = \min\{2, \sigma_2\}$.
Then,
we have
$\frac{n}{2}\left( \frac{1}{q}-\frac{1}{2} \right)+\frac{s}{2} <1$
(see Remark \ref{rem:lem:Duha} below),
and hence,
\begin{align}
    I_2
    &\lesssim
     \langle t\rangle^{\frac{n}{2}\left(\frac{1}{r}-\frac{1}{2}\right)+\frac{s}{2}} \int_{\frac{t}{2}}^{t}
     \left(
        \langle t-\tau \rangle^{
            -\frac{n}{2}\left( \frac{1}{q}-\frac{1}{2}
        \right)-\frac{s}{2}}
        \left\| \psi(\tau) \right\|_{\dot{B}^0_{q,2}}
    \right.\\
    &\qquad\qquad\qquad\qquad\qquad 
    \left.
        +
        e^{-\frac{t-\tau}{2}}
        \langle t-\tau \rangle^{\delta_2}
        \left\| \chi_{>1}(\nabla) \psi(\tau) \right\|_{\dot{B}^{s-1}_{2,2}}
    \right) \,d\tau\\
    &\lesssim
    \langle t\rangle^{\frac{n}{2}\left(\frac{1}{r}-\frac{1}{2}\right)+\frac{s}{2}}
    \int_{\frac{t}{2}}^{t}
    \left(
        \langle t-\tau \rangle^{
            -\frac{n}{2}\left( \frac{1}{q}-\frac{1}{2}
        \right)-\frac{s}{2}}
        \langle \tau \rangle^{
            -\frac{n}{2}\left( \frac{p}{r}-\frac{1}{q} \right)}
    \right.\\
    &\qquad\qquad\qquad\qquad\qquad 
    \left.
        +
        e^{-\frac{t-\tau}{2}}
        \langle t-\tau \rangle^{\delta_2}
        \langle \tau \rangle^{-\eta}
    \right) \,d\tau \| \psi \|_{Y(T)}\\
    &\lesssim
    \langle t\rangle^{-\frac{n}{2r}(p-1)+1}
    \| \psi \|_{Y(T)} \\
    &\lesssim
    \| \psi \|_{Y(T)}
    \langle T\rangle^{\left( 1-\frac{n}{2 r}(p-1) \right)_+}.
\end{align}

\noindent
\textbf{(Step 2)}
Next, we estimate
\begin{align}
    \left\|\int_{0}^{t} \mathcal{D}(t-\tau) \psi(\tau) \,d \tau \right\|_{\dot{B}^0_{r,2}}
    &\le
    \int_{0}^{\frac{t}{2}} \left\| \mathcal{D}(t-\tau) \psi(\tau)\right\|_{\dot{B}^0_{r,2}} \,d \tau
    +\int_{\frac{t}{2}}^{t}\left\| \mathcal{D}(t-\tau) \psi(\tau) \right\|_{\dot{B}^0_{r,2}} \,d \tau \\
    &=: I_3 + I_4.
\end{align}
Applying Proposition \ref{prop:lplq} with
$p=r,q=\sigma_1,s_1=s_2=0$,
we have
\begin{align}
    I_3
    &\lesssim 
    \int_{0}^{\frac{t}{2}}\left(
        \langle t-\tau
        \rangle^{-\frac{n}{2}\left(\frac{1}{\sigma_{1}}-\frac{1}{r}\right)}
        \|\psi(\tau)\|_{\dot{B}^0_{\sigma_{1},2}}
        +e^{-\frac{t-\tau}{2}}
        \langle t-\tau
        \rangle^{\delta_{r}}
        \left\| \chi_{>1}(\nabla) \psi(\tau)
        \right\|_{\dot{B}^{\beta-1}_{r,2}}\right)\, d \tau.
\end{align}
Now, we estimate the second term
$\left\| \chi_{>1}(\nabla) \psi(\tau) \right\|_{\dot{B}^{\beta-1}_{r,2}}$.
From the assumption (iii) of Theorem \ref{thm:lwp}, we have
$\min \{ \frac{1}{\sigma_2}, \frac{1}{2}-\frac{s-1}{n} \} \le \frac{1}{r} - \frac{\beta-1}{n}$.
Therefore, by the Sobolev embedding, we obtain 
\begin{align}
	\| \chi_{>1}(\nabla) \psi (\tau) \|_{\dot{B}^{\beta-1}_{r,2}}
	\lesssim
			\| \chi_{>1}(\nabla) \psi(\tau) \|_{\dot{B}^{-n(\frac{1}{2}-\frac{1}{r})+(s-1)}_{r,2}}
	\lesssim
			\| \chi_{>1}(\nabla) \psi (\tau) \|_{\dot{B}^{s-1}_{2,2}}
	\lesssim \| \psi \|_{Y(T)}
\end{align}
if $\frac{1}{2}-\frac{s-1}{n} \le \frac{1}{r}-\frac{\beta-1}{n}$
(i.e., $(2n-1)(\frac{1}{2}-\frac{1}{r}) \le s$),
and
\begin{align}
	\| \chi_{>1}(\nabla) \psi (\tau) \|_{\dot{B}^{\beta-1}_{r,2}}
	\lesssim \| \chi_{>1}(\nabla) \psi(\tau) \|_{\dot{B}^{-n(\frac{1}{\sigma_2}-\frac{1}{r})}_{r,2}}
	\lesssim \| \chi_{>1}(\nabla) \psi (\tau) \|_{\dot{B}^{0}_{\sigma_2,2}}
	\lesssim \| \psi \|_{Y(T)}
\end{align}
if $\frac{1}{\sigma_2} \le \frac{1}{r} - \frac{\beta-1}{n}$
(i.e., $\beta \le 1$ and $p\le \frac{2n}{n-2s} \left( \frac{1}{r} + \frac{1-\beta}{n} \right)$ if $2s < n$).
Thus, we obtain
\begin{align}
    I_3
    &\lesssim
    \langle t\rangle^{
        -\frac{n}{2}\left(\frac{1}{\sigma_{1}}-\frac{1}{r}\right)}
    \int_{0}^{\frac{t}{2}}
    \langle\tau\rangle^{
        -\frac{n}{2}\left(\frac{p}{r}-\frac{1}{\sigma_1}\right)}
    \,d \tau \|\psi\|_{Y(T)}
    + 
    e^{-t/4}
    \int_0^{\frac{t}{2}}
     \langle t-\tau
        \rangle^{\delta_{r}} \,d\tau 
    \|\psi \|_{Y(T)}\\
    &\lesssim
    \| \psi \|_{Y(T)}
    \langle T\rangle^{\left( 1-\frac{n}{2 r}(p-1) \right)_+}.
\end{align}
Finally, for $I_4$, we apply
Proposition \ref{prop:lplq} with
$p=r, q = \sigma_1, s_1=s_2 = 0$
to obtain
\begin{align}
    I_4
    &\lesssim
    \int_{\frac{t}{2}}^{t}
    \left(
    \langle t-\tau\rangle^{
        -\frac{n}{2}\left(\frac{1}{\sigma_1}-\frac{1}{r}\right)}
    \left\|\psi(\tau) \right\|_{\dot{B}^0_{\sigma_1,2}}
    +
        e^{-\frac{t-\tau}{2}}
        \langle t-\tau\rangle^{\delta_{r}}
        \left\| \chi_{>1}(\nabla) \psi(\tau) \right\|_{\dot{B}^{\beta-1}_{r,2}}
    \right) \,d \tau \\
    &\lesssim
    \int_{\frac{t}{2}}^{t}
    \langle t-\tau\rangle^{
        -\frac{n}{2}\left(\frac{1}{\sigma_1}-\frac{1}{r}\right)}
    \langle\tau\rangle^{
        -\frac{n}{2}\left(\frac{p}{r}-\frac{1}{\sigma_1}\right)}
    \,d \tau \|\psi\|_{Y(T)}
    +
    \int_{\frac{t}{2}}^{t}
    e^{-\frac{t-\tau}{2}}
        \langle t-\tau\rangle^{\delta_{r}} \,d\tau
    \|\psi \|_{Y(T)} \\
    &\lesssim
    \| \psi \|_{Y(T)}
    \langle T\rangle^{\left( 1-\frac{n}{2 r}(p-1) \right)_+}.
\end{align}
This completes the proof.
\end{proof}

\begin{remark}\label{rem:lem:Duha}
Let us check the condition
$\frac{n}{2}\left( \frac{1}{q}-\frac{1}{2} \right)+\frac{s}{2} <1$
which appeared in the estimate of
$I_2$
of the above proof in the case
$0 < s \le 1$,
where
$q = \min\{2, \sigma_2\}$.
It is obvious when
$\sigma_2 \ge 2$,
and it suffices to consider the case
$\sigma_2 < 2$,
that is,
$2s < n$
and
$\sigma_2 = \frac{2n}{p(n-2s)}$.
In this case, the above condition is 
equivalent with
$p < 1+\frac{4}{n-2s}$,
which is assured from the assumption of
Theorem \ref{thm:lwp}.
\end{remark}

Next, we give a nonlinear estimate.
\begin{lemma}\label{lem:nonlin}
Under the assumptions of Theorem \ref{thm:lwp},
we have
\begin{align}
    \| u^p \|_{Y(T)}
    &\lesssim \| u \|_{X(T)}^p,\\
    \| u^p - v^p \|_{Y(T)} &\lesssim \| u - v \|_{X(T)} ( \| u \|_{X(T)} + \| v \|_{X(T)} )^{p-1}
\end{align}
for
$u, v \in X(T)$,
where the implicit constants are independent of
$T > 0$.
\end{lemma}
\begin{proof}
We remark that,
by Lemma \ref{lem:meas:N}, 
the function $u^p : [0,T] \to \dot{\mathcal{B}}^0_{\sigma_1,2}(\mathbb{R}^n)$
for $u \in X(T)$
is strongly measurable.
Therefore, it remains to prove the above estimates.

\noindent
\textbf{(Step 1)} Before going to prove the estimates, we give the following formal observation.
For functions $w_1, w_2, \ldots, w_p \in X(T)$ and $j=0,1,\ldots,p-2$, Lemma \ref{lem:frac:Leib} implies
\begin{align}\label{eq:w1-wp}
	\| w_1 w_2 \cdots w_{p-j} \|_{\dot{B}^{\mu}_{\alpha_j,2}}
	&\lesssim
	\| w_1 \cdots w_{p-j-1} \|_{\dot{B}^{\mu}_{\alpha_{j+1},2}} \| w_{p-j} \|_{L^q} \\
	&\quad
	+ \| w_{p-j} \|_{\dot{B}^{\mu}_{\beta_{j},2}} \| w_1 \cdots w_{p-j-1} \|_{L^{\frac{q}{p-j-1}}},
\end{align}
where $\mu>0, q\ge p-1$ and $\alpha_0 \ge 1$ are appropriately chosen later, and
$\alpha_j$ and $\beta_j $ are successively defined so that
the H\"{o}lder conditions are statisfied:
\begin{align}
    \frac{1}{\alpha_j}
    = \frac{1}{\alpha_{j-1}} - \frac{1}{q}
    = \frac{1}{\alpha_0} - \frac{j}{q}
    \quad (j=0,1,\ldots, p-1)
\end{align}
and
\begin{align}
    \frac{1}{\beta_j} = \frac{1}{\alpha_j} - \frac{p-j-1}{q}
    = \frac{1}{\alpha_0} - \frac{p-1}{q}
    \quad (j=0,1,\ldots,p-2).
\end{align}
Since
$\beta_j$ does not depend on $j$,
let us denote them all by $\beta_0$.
Here, we remark that
$q > \alpha_0 (p-1)$ is required to ensure $\beta_0 \in [1,\infty)$.
Then, using the estimates \eqref{eq:w1-wp} repeatedly
and applying the H\"{o}lder inequality,
we have
\begin{align}
    \left\| w_1 w_2 \cdots w_p \right\|_{\dot{B}_{\alpha_0,2}^{\mu}}
    & \lesssim
    \left\| w_1 \cdots w_{p-1} \right\|_{\dot{B}_{\alpha_1,2}^{\mu}}
    \| w_p \|_{L^q}
    + \| w_p \|_{\dot{B}_{\beta_0,2}^{\mu}}
    \| w_1 \cdots w_{p-1} \|_{L^{\frac{q}{p-1}}} \\
    &\lesssim
    \left(\left\| w_1 \cdots w_{p-2} \right\|_{\dot{B}_{\alpha_2,2}^{\mu}}
    \| w_{p-1} \|_{L^q}
    + \| w_{p-1} \|_{\dot{B}_{\beta_0,2}^{\mu}}
    \| w_1 \cdots w_{p-2} \|_{L^{\frac{q}{p-2}}}\right)
    \| w_p \|_{L^q} \\
    &\quad
    +\| w_p \|_{\dot{B}_{\beta_0,2}^{\mu}}
    \prod_{\ell \neq p} \| w_{\ell} \|_{L^q} \\
    & \lesssim \cdots \\
    &\lesssim 
    \sum_{m=1}^p \left( \| w_m \|_{\dot{B}^{\mu}_{\beta_0,2}}
    \prod_{\ell \neq m} \| w_{\ell} \|_{L^q} \right) \\
    &\lesssim 
    \sum_{m=1}^p \left( \| w_m \|_{\dot{B}^{\mu}_{\beta_0,2}}
    \prod_{\ell \neq m} \| w_{\ell} \|_{\dot{B}^0_{q,2}} \right),
\end{align}
where for the last inequality we also require $q \ge 2$.
To further estimate the right-hand side by $X(T)$-norm by applying Lemma \ref{lem:int},
we need the conditions
\begin{align}
    \frac{n}{q} &= (1-\theta) \frac{n}{r} + \theta \left( \frac{n}{2} -s \right)
    \quad \text{with some} \ \theta \in [0,1], \\
    \frac{n}{\beta_0} - \mu
    &= (1-\theta') \frac{n}{r} + \theta' \left( \frac{n}{2} -s \right)
    \quad \text{with some} \ 
    \theta' \in \left[ \frac{\mu}{s},1 \right],
\end{align}
in other words,
\begin{align}
\label{eq:cond:q1}
	\frac{n}{2}-s &\le \frac{n}{q} \le \frac{n}{r},\\
\label{eq:cond:q2}
	\frac{n}{2}-s &\le \frac{n}{\beta_0} - \mu \le \left( 1-\frac{\mu}{s} \right) \frac{n}{r} + \frac{\mu}{s} \left( \frac{n}{2}-s \right).
\end{align}
In the following, we have to choose the parameters
$\mu, q, \alpha_0$ so that the above two conditions and
\begin{equation}\label{eq:cond:q3}
	q > \alpha_0 (p-1) \quad \text{and} \quad q \ge 2
\end{equation}
hold.

\noindent
\textbf{(Step 2)}
Let us estimate
$\langle t\rangle^{\frac{n}{2}\left(\frac{p}{r}-\frac{1}{\gamma}\right)} \| u^p \|_{\dot{B}^0_{\gamma,2}}$
for
$\gamma \in [\sigma_1, \sigma_2]$.
Take
$\mu > 0$
sufficiently small and
$\alpha_0 = \frac{n\gamma}{\mu \gamma + n}$,
$q = p \gamma$.
Then, the Sobolev embedding implies
$\| u^p \|_{\dot{B}^0_{\gamma,2}} \lesssim \| u^p \|_{\dot{B}^{\mu}_{\alpha_0,2}}$.
In this case $\beta_0$ is given by
$\beta_0 = \frac{p n \gamma}{p\mu \gamma + n}$.
Note that
$\alpha_0, \beta_0 \ge 1$
holds if $\mu$ is sufficiently small.
Then, by $p \ge 2$ and $\gamma > 1$, the condition $q \ge 2$ obviously holds.
Moreover, $\alpha_0 (p-1) = \frac{n\gamma}{\mu \gamma + n} (p-1) < p \gamma$
is true.
Thus, the condition \eqref{eq:cond:q3} is fulfilled.

Next, we check the condition \eqref{eq:cond:q1}.
From the definitions of $\sigma_1$ and $\sigma_2$ in Definition \ref{def:norm}, we have
\begin{align}
    r \le p \gamma
    \begin{dcases}
    < \infty&(2s \ge n),\\
    \le \frac{2n}{n-2s} &(2s < n)
    \end{dcases}
    \quad
    \text{for}
    \quad
    \gamma \in [\sigma_1, \sigma_2].
\end{align}
This immediately implies \eqref{eq:cond:q1}.

Finally, let us check the condition \eqref{eq:cond:q2}.
In this case it is equivalent to
\[
	\frac{n}{2}-s \le \frac{n}{p\gamma}  \le \left( 1-\frac{\mu}{s} \right) \frac{n}{r} + \frac{\mu}{s} \left( \frac{n}{2}-s \right).
\]
The left inequality has been already proved above.
By taking $\mu$ further small if needed, the right one is also true,
since $\frac{n}{p\sigma_1} < \frac{n}{r}$.
Therefore, we have \eqref{eq:cond:q2}.

Consequently, we can apply the argument in Step 1 to obtain
\begin{align}
    \langle t\rangle^{\frac{n}{2}\left(\frac{p}{r}-\frac{1}{\gamma}\right)} \| u^p \|_{\dot{B}^0_{\gamma,2}}
    &\lesssim
    \left(\langle t \rangle^{\frac{n}{2}\left(\frac{1}{r}-\frac{1}{\beta_0} \right) + \frac{\mu}{2}}
    \| u \|_{\dot{B}^{\mu}_{\beta_0, 2}}
    \right)
    \left(\langle t)^{\frac{n}{2}\left(\frac{1}{r}-\frac{1}{q}\right)}\|u\|_{\dot{B}^0_{q,2}} \right)^{p-1} \\
    &\lesssim
    \| u \|_{X(T)}^p.
\end{align}

\noindent
\textbf{(Step 3)}
Next, we consider
the estimate of
$\langle t\rangle^{\eta}
\left\| u^p \right\|_{\dot{B}^{s-1}_{2,2}}$
for $s > 1$.
Take
$\mu = s-1$
and
$\alpha_0 =2$.
In this case, since $p \ge 2$ is an integer,
the condition \eqref{eq:cond:q3} reduces to just
$q > 2(p-1)$.
Moreover, by a straightforward computation, the condition \eqref{eq:cond:q2} is rewritten as
\[
	\frac{n}{s} \left( \frac{1}{2}-\frac{1}{r} \right) \le \frac{n(p-1)}{q} \le 1.
\]
Therefore, it suffices to show that the intersection of the intervals for $\frac{n(p-1)}{q}$
obtained from \eqref{eq:cond:q1}--\eqref{eq:cond:q3}
is nonempty, that is,
\begin{equation}\label{eq:cond:q4}
	\left[ \frac{n}{s}\left( \frac{1}{2} - \frac{1}{r} \right), 1 \right]
	\cap \left[ \left(\frac{n}{2}-s\right)(p-1), \frac{n}{r}(p-1) \right]
	\cap \left(0, \frac{n}{2} \right) \neq \emptyset,
\end{equation}
which is equivalent to
\[
	\left[ \max \left\{ \frac{n}{s} \left( \frac{1}{2}- \frac{1}{r} \right), \left( \frac{n}{2}-s \right) (p-1) \right\},
		\min \left\{ 1, \frac{n}{r}(p-1) \right\} \right]
		\cap \left(0, \frac{n}{2} \right)
		\neq \emptyset.
\]
By the assumptions (i) and (ii) in Theorem \ref{thm:lwp} and $s>1$,
we can easily check the condition \eqref{eq:cond:q4} holds.
Therefore, we can appropriately choose the parameter $q$ so that
\eqref{eq:cond:q1}--\eqref{eq:cond:q3} are satisfied.

Hence, we can apply Lemma \ref{lem:int} and conclude
\begin{align}
    \langle t \rangle^{\eta}
    \| u^p \|_{\dot{B}^{s-1}_{2,2}}
    &\lesssim
     \langle t \rangle^{\frac{n}{2}\left(\frac{1}{r}-\frac{1}{\beta_0} \right) + \frac{s-1}{2}}
     \| u \|_{\dot{B}^{s-1}_{\beta_0,2}}
     \left(
     \langle t \rangle^{\frac{n}{2}\left( \frac{1}{r}-\frac{1}{q}\right)}
     \| u \|_{\dot{B}_{q,2}^0}
     \right)^{p-1}
     \\
     &\lesssim
     \| u \|_{X(T)}^p.
\end{align}

\noindent
\textbf{(Step 4)}
We treat the term
$\langle t\rangle^{\eta} \left\| \chi_{>1}(\nabla) u^p \right\|_{\dot{B}^{s-1}_{2,2}}$
for
$0 < s \le 1$.
In this case, we do not use the argument of Step 1 and just apply the Sobolev embedding.
First, we claim that there exists $\mu \in (1,2]$ such that the following relations hold:
\begin{align}
    &\frac{n}{2} \le \frac{n}{\mu} \le \frac{n}{2} + 1-s,\\
    &\mu p \ge 2,\\
    &\frac{n}{\mu p} = (1-\theta) \frac{n}{r} + \theta \left( \frac{n}{2} -s \right)
    \quad \text{with some} \ 
    \theta \in [0,1].
\end{align}
Indeed, the conditions above can be rewritten as
\begin{align}
    \frac{n}{2p} \le \frac{n}{\mu p} & < \frac{n}{p},\\
    \frac{n}{2p} \le \frac{n}{\mu p} &\le \frac{n}{2p} + \frac{1-s}{p},\\
    \frac{n}{\mu p} &\le \frac{n}{2}, \\
    \frac{n}{2} -s \le \frac{n}{\mu p} &\le \frac{n}{r}.
\end{align}
The existence of such $\mu$ is equivalent to
\begin{align}
	\left[
		\max\left\{ \frac{n}{2p}, \frac{n}{2} -s \right\},
		\min \left\{ \frac{n}{2p} + \frac{1-s}{p}, \frac{n}{r} \right\}
	\right]
	\cap
	\left( 0, \frac{n}{p} \right)
	\neq \emptyset.
\end{align}
By a simple calculation, we see that this is equivalent to
\[
    \frac{r}{2} \le p \begin{dcases}
        < \infty &( 2s \ge n),\\
        \le 1 + \frac{2}{n-2s} &(2s < n)
    \end{dcases}
    \quad \text{and} \quad 
    p < \frac{2n}{n-2s} \ (2s < n),
\]
which hold from the assumptions of Theorem \ref{thm:lwp},
since
$\min \{ \frac{r}{2}, 1 + \frac{r}{2s} - \frac{1}{s} \} = \frac{r}{2}$
when $s \le 1$,
and
$p < 1+\frac{n}{n-2s}$ if $2s < n$
(see also Remark \ref{rem:lwp} (ii)).
Thus, the claim is proved.

Let $\mu$ be as above.
Since $u \in X(T)$,
by Lemma \ref{lem:int},
we have
$u \in L^{\mu p}(\mathbb{R}^n)$ and hence
$u^p \in L^{\mu}(\mathbb{R}^n)$.
Then, by the embeddings
$L^{\mu}(\mathbb{R}^n) \subset B^0_{\mu, 2}(\mathbb{R}^n)
\subset
B^{s-1}_{2,2}(\mathbb{R}^n)$,
it is obvious that
$\chi_{>1}(\nabla) u^p \in \mathcal{S}_h'(\mathbb{R}^n)$
and
we have the estimates
\begin{equation}\label{eq:up:s<1}
    \| \chi_{>1}(\nabla) u^p \|_{\dot{B}^{s-1}_{2,2}}
    \lesssim
    \| u^p \|_{B^{s-1}_{2,2}}
    \lesssim
    \| u^p \|_{L^{\mu}}
    \lesssim
    \| u \|_{L^{\mu p}}^p
    \lesssim
    \| u \|_{\dot{B}^0_{\mu p,2}}^p.
\end{equation}
Thus, we conclude
\begin{align}
    \langle t \rangle^{\eta} \| \chi_{>1}(\nabla) u^p \|_{\dot{B}^{s-1}_{2,2}}
    &\lesssim
    \langle t \rangle^{\frac{s-1}{2} + \frac{n}{2}\left( \frac{p}{r}-\frac{1}{2} \right)} \| u \|_{\dot{B}^0_{\mu p,2}}^p \\
    &\lesssim
    \langle t \rangle^{\frac{n}{2}\left( \frac{p}{r}-\frac{1}{\mu}\right)}
    \| u \|_{\dot{B}^0_{\mu p,2}}^p \\
    &\lesssim
    \left( \langle t \rangle^{\frac{n}{2}\left(\frac{1}{r}-\frac{1}{\mu p}\right)} \| u \|_{\dot{B}^0_{\mu p,2}} \right)^p \\
    &\lesssim
    \| u \|_{X(T)}^p.
\end{align}

\noindent
\textbf{(Step 5)}
Finally, we prove the estimate of
$\| u^p - v^p \|_{Y(T)}$.
For the term
$\langle t\rangle^{\frac{n}{2}\left(\frac{p}{r}-\frac{1}{\gamma}\right)} \| u^p-v^p \|_{\dot{B}^0_{\gamma,2}}$
with
$\gamma \in [\sigma_1, \sigma_2]$,
we express the difference $u^p-v^p$ as
\begin{align}
	u^p-v^p
	&= u^{p-1}(u-v) + u^{p-2} v (u-v) + \cdots + v^{p-1} (u-v)
\end{align}
and apply the argument of Steps 1--2 to each term.
Then, we have
\[
	\langle t\rangle^{\frac{n}{2}\left(\frac{p}{r}-\frac{1}{\gamma}\right)} \| u^p-v^p \|_{\dot{B}^0_{\gamma,2}}
	\lesssim
	\| u- v \|_{X(T)} \left( \|u \|_{X(T)} + \| v \|_{X(T)} \right)^{p-1}.
\]
The estimate of the term
$\langle t\rangle^{\eta}
\left\| u^p-v^p \right\|_{\dot{B}^{s-1}_{2,2}}$
for $s > 1$
can be obtained in the same way.
The estimate of the term
$\langle t\rangle^{\eta} \left\| \chi_{>1}(\nabla) u^p \right\|_{\dot{B}^{s-1}_{2,2}}$
for
$0 < s \le 1$
can be proved by combining the estimate \eqref{eq:up:s<1} with
$\| u^p-v^p \|_{L^{\mu}} \lesssim \| u-v \|_{L^{\mu p}} \left( \| u \|_{L^{\mu p}} + \| v \|_{L^{\mu p}} \right)^{p-1}$.
This completes the proof.
\end{proof}

To prove Theorem \ref{thm:lwp},
we apply the contraction mapping principle.
Let $T \in (0,\infty]$, $M>0$
and define
\begin{align}\label{eq:xtm}
    X(T, M)
    = \{ \phi \in X(T) ; \| \phi \|_{X(T)} \le M \}
\end{align}
with the metric
$d(\phi,\psi) := \| \phi - \psi \|_{X(T)}$.
We first state the following.
\begin{lemma}\label{lem:XTM}
$X(T,M)$ equipped with the metric $d(\cdot, \cdot)$ above
is a complete space.
\end{lemma}
We postpone the proof of this lemma to
Appendix \ref{app:E}
and continue the proof of Theorem \ref{thm:lwp}.

\begin{proof}[Proof of Theorem \ref{thm:lwp}]
Let us prove the existence of the solution
to \eqref{eq:ndw}.
Let $T \in (0,1)$.
For $u \in X(T)$,
we define the mapping
\begin{align}\label{eq:psi}
    \Psi(u)(t):=\left(\partial_{t}+1\right) \mathcal{D}(t)  u_{0}+\mathcal{D}(t)  u_{1}
    +\int_{0}^{t} \mathcal{D}(t-\tau) u(\tau)^p \,d \tau.
\end{align}
First, let
\[
	\| (u_0, u_1) \|_Z := \| u_0 \|_{
    \dot{B}^{s}_{2,2} \cap \dot{B}_{r,2}^0
    }
    + \| \chi_{>1}(\nabla) u_0 \|_{\dot{B}_{r,2}^{\beta}}
    +
    \| u_1 \|_{\dot{B}^{s-1}_{2,2}\cap \dot{B}^{0}_{r,2}}
    + \| \chi_{>1}(\nabla) u_1 \|_{\dot{B}_{r,2}^{\beta-1}}.
\]
By Proposition \ref{prop:lplq}, we have
\begin{align}\label{eq:lin:est}
\quad
    \| \left(\partial_{t}+1\right) \mathcal{D}(t)  u_{0}+\mathcal{D}(t)  u_{1} \|_{X(T)}
    \le C_0 \| (u_0, u_1) \|_Z
\end{align}
with some constant $C_0 > 0$
independent of $T$.
Let
\begin{align}
    M := 2 C_0 \| (u_0, u_1) \|_Z.
\end{align}
Then, Lemmas \ref{lem:Duha} and \ref{lem:nonlin} yield
\begin{align}\label{eq:apriori}
    \|\Psi(u)\|_{X(T)} \leq \frac{M}{2}+C_{1} T M^{p}
\end{align}
and
\begin{align}\label{eq:apriori:diff}
    \|\Psi(u)-\Psi(v)\|_{X(T)} \leq C_{2} T(2 M)^{p-1}\|u-v\|_{X(T)}
\end{align}
for any $u,v \in X(T,M)$
with some constants $C_1, C_2 > 0$.
Therefore, taking
$T \in (0,1)$
sufficiently small, we see that
$\Psi$
is a contraction mapping on $X(T,M)$.
Thus, the fixed point theorem implies
there exists a unique
$u \in X(T,M)$
such that $\Psi (u) = u$,
that is,
$u$
satisfies the integral equation \eqref{eq:def:sol}.

Next, we show the continuity in $t$ of $u$.
For the linear part, by Proposition \ref{prop:lplq}, we see that
\begin{align}\label{eq:lin:sp}
    \left(\partial_{t}+1\right) \mathcal{D}(t)  u_{0}+\mathcal{D}(t) u_{1}
    \in
    C([0,T]; \dot{\mathcal{B}}^s_{2,2}(\mathbb{R}^n) \cap \dot{\mathcal{B}}^0_{r,2}(\mathbb{R}^n)
    ).
\end{align}
Indeed, first, the boundedness of the
$\dot{\mathcal{B}}^s_{2,2}(\mathbb{R}^n) \cap \dot{\mathcal{B}}^0_{r,2}(\mathbb{R}^n)$-norm
follows from the estimate \eqref{eq:lin:est}.
Next, we show the
continuity in time in \eqref{eq:lin:sp}.
For any fixed
$t_0 \ge 0$,
from the definition of
$\mathcal{D}(t)$
(see \eqref{eq:Dt} and \eqref{eq:L}),
it is easy to see that
$\lim_{t\to t_0}
\| \Delta_j (\mathcal{D}(t) u_1 - \mathcal{D}(t_0) u_1) \|_{L^r} = 0$
holds for each
$j \in \mathbb{Z}$.
This and Proposition \ref{prop:lplq}
enable us to apply the Lebesgue
dominated convergence theorem to obtain
$\lim_{t\to t_0}
\| \mathcal{D}(t) u_1 - \mathcal{D}(t_0) u_1 \|_{\dot{B}^{0}_{r,2}} = 0$.
The other norms can be treated in the same way.
A similar argument with the integral equation
\eqref{eq:def:sol} shows the continuity of
the Duhamel term in time, that is, we have
$u \in C([0,T); \dot{\mathcal{B}}^s_{2,2}(\mathbb{R}^n) \cap \dot{\mathcal{B}}^0_{r,2}(\mathbb{R}^n))$.

Finally, we prove the uniqueness of
the solution in the space
$X(T)$.
Let $u$ and $v$ be solutions to \eqref{eq:ndw}
in
$X(T)$.
Set
$\tilde{M} := \| u\|_{X(T)} + \| v \|_{X(T)}$.
Then, by Lemmas \ref{lem:Duha} and \ref{lem:nonlin},
for any $T_1 \in (0,T)$,
we have
\begin{align}
    \| u - v \|_{X(T_1)}
    &=
    \left\|\int_{0}^{t} \mathcal{D}(t-\tau)(u^p-v^p) \,d\tau \right\|_{X(T_1)} \\
    &\le
    C_3 T_1 \| u - v \|_{X(T_1)}
    \left( \|u\|_{X(T_1)}+\|v\|_{X(T_1)} \right)^{p-1}\\
    &\le
    C_3 T_1 \tilde{M}^{p-1} \| u - v \|_{X(T_1)}
\end{align} 
with some constant $C_3$
independent of $T_1 \in (0,T)$.
Hence, taking $T_1$
sufficiently small, we have
$\| u -v \|_{X(T_1)} = 0$,
that is,
$u \equiv v$ on $(0,T_1)$.
Repeating this argument, we conclude
$u \equiv v$ in $(0,T)$.
This completes the proof of Theorem \ref{thm:lwp}.
\end{proof}

Finally, we give a proof of Theorem \ref{thm:gwp}.
\begin{proof}[Proof of Theorem \ref{thm:gwp}]
We remark that, due to the derivative loss,
we cannot apply the usual extension argument.
To avoid it, following the argument by
\cite{HaKaNa04DIE},
we directly construct the global solution 
by the contraction mapping principle
on the interval $[0,\infty)$.

The argument is similar to that of
Theorem \ref{thm:lwp}.
Let
$p \ge 1+ \frac{2r}{n}$,
and let
$\varepsilon_0 > 0$
be a constant determined later.
Define
$M:= 2C_0 \varepsilon_0$,
where $C_0$ is the constant in
the linear estimate \eqref{eq:lin:est}.
We consider
$X(\infty,M)$ defined by \eqref{eq:xtm}.
We again consider the mapping
$\Psi$ defined by \eqref{eq:psi}.
Then, in this case,
Lemmas \ref{lem:Duha} and \ref{lem:nonlin} and the assumption \eqref{eq:varepsilon0} imply
\begin{align}
    \|\Psi(u)\|_{X(\infty)} \leq \frac{M}{2}+C_{1} M^{p}
\end{align}
and
\begin{align}
    \|\Psi(u)-\Psi(v)\|_{X(\infty)} \leq C_{2} (2 M)^{p-1}\|u-v\|_{X(\infty)}
\end{align}
instead of \eqref{eq:apriori}
and \eqref{eq:apriori:diff}, respectively.
Then, taking
$\varepsilon_0$
sufficiently small
depending on $C_0, C_1, C_2, p$,
we see that
$\Psi$
is a contraction mapping on
$X(\infty,M)$.
Hence, we find the solution
$u \in X(\infty,M)$,
and by the integral equation \eqref{eq:def:sol},
we have
$u \in C([0,\infty); \dot{\mathcal{B}}^s_{2,2}(\mathbb{R}^n)
\cap \dot{\mathcal{B}}^0_{r,2}(\mathbb{R}^n))$.
The uniqueness of solution in
$X(\infty)$
can be proved in the same manner as that of 
Theorem \ref{thm:lwp}.
\end{proof}

\appendix
\section{Proof of Proposition \ref{prop:lplq}}
We first note that,
if $g \in \mathcal{S}_h'(\mathbb{R}^n)$, then
$\mathcal{D}(t) g \in \mathcal{S}_h'(\mathbb{R}^n)$ for all $t > 0$.
Indeed, for any fixed $t > 0$ and for sufficiently small $a > 0$, we write
\begin{align}
    \chi_{\le a}(\nabla) \mathcal{D}(t) g
    = \mathcal{D}(t) \chi_{\le 1}(\nabla) \chi_{\le a}(\nabla) g
    = e^{-t/2} \mathcal{F}^{-1} \left[ L(t,\cdot) \chi_{\le 1}(\cdot) \right] \ast \chi_{\le a}(\nabla) g.
\end{align}
Since $L(t,\cdot) \chi_{\le 1}(\cdot) \in C_0^{\infty} (\mathbb{R}^n)$,
we apply the Young inequality to obtain
\begin{align}
    \| \chi_{\le a}(\nabla) \mathcal{D}(t) g \|_{L^{\infty}}
    \le
   e^{-t/2}  \| \mathcal{F}^{-1} \left[ L(t,\cdot) \chi_{\le 1}(\cdot) \right] \|_{L^1}
   \| \chi_{\le a}(\nabla) g \|_{L^{\infty}}
   \rightarrow 0 \quad (a \to 0).
\end{align}
This implies $\mathcal{D}(t) g \in \mathcal{S}_h'(\mathbb{R}^n)$.

To prove Proposition \ref{prop:lplq}, it suffices to show
\begin{align}\label{eq:App:A1}
    \qquad
    2^{k s_1} \| \Delta_k \mathcal{D}(t) g \|_{L^p}
    &\lesssim
    \begin{dcases}
     \langle t\rangle^{-\frac{n}{2}\left(\frac{1}{q}-\frac{1}{p}\right)-\frac{s_{1}-s_{2}}{2}}
     2^{ks_2}
    \left\| \Delta_k g \right\|_{L^q}
    &(k \le 0),\\
    e^{-\frac{t}{2}} \langle t\rangle^{\delta_{p}}
    2^{k(s_1+\beta-1)}
    \left\| \Delta_k g\right\|_{L^p}
    &(k \ge 0).
    \end{dcases}
\end{align}
Replacing
$\Delta_k$
by
$\widetilde{\Delta}_k \Delta_k$,
and using the completely same argument
as in
\cite[Propositions 2.4 and 2.5]{IkInOkWa19},
we can prove the following two lemmas:
\begin{lemma}
Let
$k \le 0$,
$1 \le q \le p \le \infty$,
and
$s_1 \ge s_2$.
Then,
\begin{align}
    2^{ks_1} \left\| \widetilde{\Delta}_k \mathcal{D}(t) \Delta_k g \right\|_{L^{p}}
    \lesssim
    \langle t\rangle^{-\frac{n}{2}\left(\frac{1}{q}-\frac{1}{p}\right)-\frac{s_{1}-s_{2}}{2}}
    2^{ks_2}
    \left\| \Delta_k g\right\|_{L^{q}},
\end{align}
where the implicit constant
is independent of
$k$.
\end{lemma}
\begin{lemma}
Let
$k \ge 0$,
$1<p<\infty$,
$\beta = (n-1) |\frac{1}{2}-\frac{1}{p}|$.
Then, there exists
$\delta_p > 0$
such that
\begin{align}
    \left\| \widetilde{\Delta}_k \mathcal{D}(t) \Delta_k g\right\|_{L^{p}}
    \lesssim
    e^{-\frac{t}{2}}
    \langle t\rangle^{\delta_{p}}
    2^{k(\beta-1)}
    \| \Delta_k g \|_{L^p},
\end{align}
where the implicit constant
is independent of
$k$.
\end{lemma}
These estimates immediately
imply \eqref{eq:App:A1}.


\section{Proof of the statements in Remark \ref{rem:frac:Leib}}\label{app:FraLei}
In this appendix, we show the statements in Remark \ref{rem:frac:Leib}.

\noindent
\textup{(i)}
First, under the assumptions
$\alpha > 0$, $1 \le r, \sigma, p_1, p_2, q_1, q_2 \le \infty$
satisfying
$\frac{1}{r} = \frac{1}{p_1}+\frac{1}{q_1} = \frac{1}{p_2} + \frac{1}{q_2}$
and
$f \in \dot{\mathcal{B}}^{\alpha}_{p_1,\sigma}(\mathbb{R}^n) \cap L^{q_2}(\mathbb{R}^n)$
and
$g \in \dot{\mathcal{B}}^{\alpha}_{p_2,\sigma}(\mathbb{R}^n) \cap L^{q_1}(\mathbb{R}^n)$,
we prove 
$fg \in \mathcal{S}'(\mathbb{R}^n)$.

Since $f \in L^{q_2}(\mathbb{R}^n)$ and 
$g \in L^{q_1}(\mathbb{R}^n)$,
the product $fg$ makes sense as a measurable function.
Decompose $f$ as
\[
    f = f_L + f_H :=
    \Delta_{\le 0} f + \Delta_{>0} f \quad \text{in} \ \mathcal{S}'(\mathbb{R}^n)
\]
and $g=g_L + g_H$ in the same way.
Then, we can immediately see that
\begin{align*}
    f_L \in L^{\infty}(\mathbb{R}^n) \cap L^{q_2}(\mathbb{R}^n),\quad
    f_H \in B^{\alpha}_{p_1,\sigma}(\mathbb{R}^n) \cap L^{q_2}(\mathbb{R}^n),\\
    g_L \in L^{\infty}(\mathbb{R}^n) \cap L^{q_1}(\mathbb{R}^n),\quad
    g_H \in B^{\alpha}_{p_2,\sigma}(\mathbb{R}^n) \cap L^{q_1}(\mathbb{R}^n).
\end{align*}
In fact, for $f_L$, we have
$\| f_L \|_{L^{\infty}} = \| \mathcal{F}^{-1} [ \chi_{\le 1} ] \ast f \|_{L^{\infty}}
\lesssim \| \mathcal{F}^{-1} [\chi_{\le 1}] \|_{L^{q_2'}} \| f \|_{L^{q_2}}$
and
$\| f_L \|_{L^{q_2}} \lesssim  \| \mathcal{F}^{-1} [\chi_{\le 1}] \|_{L^{1}} \| f \|_{L^{q_2}}$.
For $f_H$, we have
$f_H = f - f_L \in L^{q_2}(\mathbb{R}^n)$
and
$\| f_H \|_{\dot{B}^{\alpha}_{p_1,\sigma}} \sim \| f_H \|_{B^{\alpha}_{p_1, \sigma}}$.
Moreover, we remark that
$B^{\alpha}_{p_1,\sigma}(\mathbb{R}^n) \subset L^{p_1}(\mathbb{R}^n)$
\footnote{%
By
$\| f \|_{L^{p_1}} \le \| \Delta_{\le 0} f \|_{L^{p_1}} + \| \sum_{j\ge 1} \Delta_j f \|_{L^{p_1}}$
and
$\| \sum_{j\ge 1} \Delta_j f \|_{L^{p_1}}
\le \sum_{j\ge 1} \| \Delta_j f \|_{L^{p_1}}
\le \| \{ 2^{-j \alpha} \}_{j \ge 1} \|_{l^{\sigma'}} \| \{ 2^{j\alpha} \| \Delta_j f \|_{L^{p_1}} \}_{j \ge 1} \|_{l^{\sigma}}
\lesssim \| f \|_{B^{\alpha}_{p_1,\sigma}}$.
}%
implies
$f_H \in L^{p_1}(\mathbb{R}^n)$
and
$g_H \in L^{p_2}(\mathbb{R}^n)$.
Thus, we have the pointwise decomposition
\begin{equation}\label{eq:app:b1}
    f g = (f_L+f_H)(g_L+g_H)
    = f_L g_L + f_L g_H + f_H g.
\end{equation}
Next, we estimate the right-hand side term by term.
The H\"{o}lder inequality implies
\[
    \| f_L g_L \|_{L^{q_1}}
    \le \| f_L \|_{L^{\infty}} \| g_L \|_{L^{q_1}} < \infty.
\]
The other terms can be estimated as
\begin{align}
    \| f_L g_H \|_{L^r} \le \| f_L \|_{L^{q_2}} \| g_H \|_{L^{p_2}} < \infty,\\
    \| f_H g \|_{L^r} \le \| f_H \|_{L^{p_1}} \| g \|_{L^{q_1}} < \infty.
\end{align}
Therefore, each term of the right-hand side of \eqref{eq:app:b1} belongs to $\mathcal{S}'(\mathbb{R}^n)$,
and so is $fg$.

\noindent
\textup{(ii)}
Next, we show that
$f g \in \mathcal{S}'_h (\mathbb{R}^n)$
if we additionally suppose
$r < \infty$ and $(q_1, q_2) \neq (\infty,\infty)$.
Let us suppose $q_1 \neq \infty$.
Then, the above estimates imply
$f_L g_L \in L^{q_1}(\mathbb{R}^n)$
and
$f_L g_H, f_H g \in L^r(\mathbb{R}^n)$.
Since $q_1, r < \infty$,
these three terms belong to $\mathcal{S}_h'(\mathbb{R}^n)$,
and so is $fg$.

\noindent
\textup{(iii)}
Finally, we show that the paraproduct decomposition
\begin{equation}\label{eq:app:b2}
	fg = \dot{T}_f g + \dot{T}_g f + \dot{R}(f, g) \quad \text{in} \ \mathcal{S}'(\mathbb{R}^n)
\end{equation}
holds, under the assumptions of Lemma \ref{lem:frac:Leib} with $r, q_1, q_2 < \infty$,
where
\[
	\dot{T}_f g := \sum_{j \in \mathbb{Z}} \Delta_{\le j-2} f \Delta_{j} g,\quad
	\dot{R}(f,g) := \sum_{\substack{k, j \in \mathbb{Z} \\ |k-j| \le 1}} \Delta_k f \Delta_j g.
\] 
Set
\[
	V_j f := \sum_{k=-j}^j \Delta_k f
\]
for $j \in \mathbb{N}$.
Then, we have the decomposition
\begin{equation}\label{eq:app:b3}
	V_j f V_j g = \dot{T}_{V_j f} V_j g + \dot{T}_{V_j g} V_j f + \dot{R}(V_j f, V_j g),
\end{equation}
where each term consists of finite sum.
Then, to prove \eqref{eq:app:b2}, it suffices to show
\begin{align}\label{eq:app:b4}
	V_j f V_j g
		&\to fg \quad \text{in} \ \mathcal{S}'(\mathbb{R}^n), \\
\label{eq:app:b5}
	\dot{T}_{V_j f} V_j g
		&\to \dot{T}_f g,\quad
	\dot{T}_{V_j g} V_j f \to \dot{T}_g f, \quad
		\dot{R}(V_j f, V_j g) \to \dot{R}(f, g) \quad \text{in} \ \mathcal{S}'(\mathbb{R}^n)
\end{align}
as $j \to \infty$.

To prove \eqref{eq:app:b4}, we decompose $V_j f V_j g$ into
\[
	V_j f V_j g = V_j f_L V_j g_L + V_j f_L V_j g_H + V_j f_H V_j g_L + V_j f_H V_j g_H.
\]
First, we remark that
$\| V_j h \|_{L^p} \lesssim \| h \|_{L^p}$
holds for any
$j\in \mathbb{N}$, 
$1 \le p \le \infty$ and $h \in L^p(\mathbb{R}^n)$,
where the implicit constant is independent of $j$.
In fact, we have
$V_j h = \chi_{\le 2^j}(\nabla) h - \chi_{\le 2^{-j-1}}(\nabla ) h$
and
$\| \chi_{\le a}(\nabla) h \|_{L^p} \lesssim \| \mathcal{F}^{-1} [ \chi ] \|_{L^1} \| h \|_{L^p}$.
Moreover, we compute
\[
	\| V_j f_L - f_L \|_{L^{\infty}}
	= \left\| \sum_{k=-j}^0 \Delta_k f - \sum_{k \le 0} \Delta_k f  \right\|_{L^{\infty}}
	= \left\| \sum_{k\le -j-1} \Delta_k f \right\|_{L^{\infty}}
	\to 0
\]
as $j \to \infty$, since $f \in \mathcal{S}'_h(\mathbb{R}^n)$.
We also have
\begin{align*}
	\| V_j f_H - f_H \|_{L^{p_1}}
	&\le \left\| \sum_{k \ge j+1} \Delta_k f \right\|_{L^{p_1}} \\
	&\lesssim
	\| \{ 2^{-\alpha k} \}_{k \ge j+1} \|_{l^{\sigma'}} \| \{ 2^{\alpha k} \| \Delta_k f \|_{L^{p_1}} \}_{k \ge j+1} \|_{l^{\sigma}} \\
	&\lesssim
	\| \{ 2^{-\alpha k} \}_{k \ge j+1} \|_{l^{\sigma'}}
	\| f \|_{\dot{B}^{\alpha}_{p_1,\sigma}}
	\to 0
\end{align*}
as $j \to \infty$.
In the same way, we can show
$\| V_j g_H - g_H \|_{L^{p_2}} \to 0$ as $j\to \infty$.
These estimates imply
\begin{align*}
	\| V_j f_L V_j g_L - f_L g_L \|_{L^{\infty}}
	&\le
	\| (V_j f_L - f_L) V_j g_L \|_{L^{\infty}} + \| f_L ( V_j g_L - g_L ) \|_{L^{\infty}} \\
	&\le \| V_j f_L - f_L \|_{L^{\infty}} \| V_jg_L \|_{L^{\infty}}
		+ \| f_L \|_{L^{\infty}} \| V_j g_L - g_L \|_{L^{\infty}}
	\to 0
\end{align*}
and
\begin{align*}
	\| V_j f_L V_j g_H - f_L g_H \|_{L^{p_2}}
	&\le 
	\| (V_j f_L - f_L) V_j g_H \|_{L^{p_2}} + \| f_L ( V_j g_H - g_H ) \|_{L^{p_2}} \\
	&\le
	\| V_j f_L - f_L \|_{L^{\infty}} \| V_j g_H \|_{L^{p_2}}
		+ \| f_L \|_{L^{\infty}} \| V_j g_H - g_H \|_{L^{p_2}}
	\to 0
\end{align*}
as $j \to \infty$.
In the same way, we can obtain
\[
	\| V_j f_H V_j g_L - f_H g_L \|_{L^{p_1}} \to 0
\]
as $j \to \infty$.
Finally, for $V_j f_H V_j g_H$, we have
\begin{align*}
	\| V_j f_H V_j g_H - f_H g_H \|_{L^r}
	&\le
	\| (V_j f_H - f_H ) V_j g_H \|_{L^r} + \| f_H ( V_j g_H - g_H) \|_{L^r} \\
	&\le
	\| V_j f_H - f_H \|_{L^{p_1}} \| V_j g_H \|_{L^{q_1}}
	+ \| f_H \|_{L^{q_2}} \| V_j g_H - g_H \|_{L^{p_2}} \to 0
\end{align*}
as $j \to \infty$.
This proves \eqref{eq:app:b4}.

To prove \eqref{eq:app:b5}, we write
\[
	\dot{T}_{V_j f} V_j g - \dot{T}_f g 
	= \dot{T}_{V_j f - f} V_j g_L
	+ \dot{T}_{V_j f - f} V_j g_H
	+ \dot{T}_f (V_j g_L - g_L)
	+ \dot{T}_f (V_j g_H - g_H).
\]
First, we estimate the term
$\dot{T}_{V_j f - f} V_j g_L$.
Since $\Delta_k V_j g_L = 0$ if $k \notin \{ -j-1, -j,\ldots, 1\}$
and $\Delta_{\le k-2} (V_j f-f) = \Delta_{\le k-2} \Delta_{\le -j-1}f $ if $k \in \{ -j-1, -j,\ldots, 1\}$,
we can write
\[
	\dot{T}_{V_j f - f} V_j g_L
	= \sum_{k \in \mathbb{Z}} \Delta_{\le k-2} (V_j f - f ) \Delta_k V_j g_L
	= \sum_{k=-j-1}^1 \Delta_{\le k-2}\Delta_{\le -j-1}f \Delta_k V_j g_L.
\]
Taking arbitrary $\mu_1 \in (q_1, \infty)$ and using the embedding
$L^{q_1}(\mathbb{R}^n) \subset \dot{\mathcal{B}}^0_{q_1,\infty}(\mathbb{R}^n) \subset
\dot{\mathcal{B}}^{-\beta}_{\mu_1, \infty}(\mathbb{R}^n)$
with $\beta = n (1/q_1 - 1/\mu_1) > 0$,
we estimate
\begin{align*}
	\| \dot{T}_{V_j f - f} V_j g_L \|_{L^{\mu_1}}
	&=
	\left\| \sum_{k=-j-1}^{1} (\Delta_{\le k-2} \Delta_{\le -j-1} f) \Delta_k V_j g_L \right\|_{L^{\mu_1}} \\
	&\le
	\sum_{k=-j-1}^{1} \| \Delta_{\le k-2} \Delta_{\le -j-1} f \|_{L^{\infty}}
		\| \Delta_k V_j g_L \|_{L^{\mu_1}} \\
	&\lesssim
	\| \Delta_{\le -j-1} f \|_{L^{\infty}}
		\| \{ 2^{k\beta} \}_{k\le 1} \|_{\ell^1}
		\| \{ 2^{-k\beta} \| \Delta_k g \|_{L^{\mu_1}} \}_{k\le 1} \|_{\ell^{\infty}} \\
	&\lesssim
	\| \Delta_{\le -j-1} f \|_{L^{\infty}} \| g \|_{\dot{B}^{-\beta}_{\mu_1, \infty}}
	\to 0
\end{align*}
as $j\to \infty$.
Here, we have used the fact $f \in \mathcal{S}_h'(\mathbb{R}^n)$.
In a smilar way, we have
\[
	\dot{T}_{V_j f - f} V_j g_H
	= \sum_{k=0}^{j+1} \Delta_{\le k-2}\Delta_{\le -j-1}f \Delta_k V_j g_L,
\]
and hence,
\begin{align*}
	\| \dot{T}_{V_j f - f} V_j g_H \|_{L^{p_2}}
	&=
	\left\| \sum_{k=0}^{j+1} (\Delta_{\le k-2} \Delta_{\le -j-1} f) \Delta_k V_j g_H \right\|_{L^{p_2}} \\
	&\le
	\sum_{k=0}^{j+1}
	\| \Delta_{\le k-2} \Delta_{\le -j-1} f \|_{L^{\infty}}
	\| \Delta_k V_j g_H \|_{L^{p_2}} \\
	&\lesssim
	\| \Delta_{\le -j-1} f \|_{L^{\infty}} 
	\| \{ 2^{-\alpha k} \}_{k \ge 0} \} \|_{\ell^{\sigma'}}
	\| \{ 2^{\alpha k} \| \Delta_k g \|_{L^{p_2}} \}_{k \ge 0} \|_{L^{p_2}} \\
	&\lesssim
	\| \Delta_{\le -j-1} f \|_{L^{\infty}} \| \{ 2^{-\alpha k} \}_{k \ge 0} \} \|_{\ell^{\sigma'}}
	\| g \|_{\dot{B}^{\alpha}_{p_2, \sigma}}
	\to 0 
\end{align*}
as $j\to \infty$.
Taking the same $\mu_1 \in (q_1, \infty)$ as above,
we have
\begin{align*}
	\| \dot{T}_f (V_j g_L - g_L) \|_{L^{\mu_1}}
	&= \left\| \sum_{k\in \mathbb{Z}} \Delta_{\le k-2} f \Delta_k \left( V_j g_L - g_L \right) \right\|_{L^{\mu_1}} \\
	&= \left\| \sum_{k\le -j} \Delta_{\le k-2} f \Delta_k \left( V_j g_L - g_L \right) \right\|_{L^{\mu_1}} \\
	&\lesssim
		\left( \sup_{k \le -j} \| \Delta_{\le k-2} f \|_{L^{\infty}} \right)
		\| \{ 2^{k\beta} \}_{k \le -j} \|_{l^1}
		\| \{ 2^{-k\beta} \| \Delta_k g \|_{L^{\mu_1}} \}_{k \le -j} \|_{l^{\infty}} \\
	&\lesssim
		\left( \sup_{k \le -j} \| \Delta_{\le k-2} f \|_{L^{\infty}} \right)
		\| \{ 2^{k\beta} \}_{k \le -j} \|_{l^1}
		\| g \|_{\dot{B}^{-\beta}_{\mu_1,\infty}}
	\to 0
\end{align*}
as $j \to \infty$.
We also have
\begin{align*}
	\| \dot{T}_f (V_j g_H - g_H) \|_{L^{r}}
	&= \left\| \sum_{k\in \mathbb{Z}} \Delta_{\le k-2} f \Delta_k \left( V_j g_H - g_H \right) \right\|_{L^{r}} \\
	&= \left\| \sum_{k \ge j} \Delta_{\le k-2} f \Delta_k \left( V_j g_H - g_H \right) \right\|_{L^{r}} \\
	&\lesssim \| f \|_{L^{q_2}} \sum_{k \ge j} \| \Delta_{k} g \|_{L^{p_2}} \\
	&\lesssim \| f \|_{L^{q_2}}
		\| \{ 2^{-k \alpha} \}_{k \ge j} \|_{l^{\sigma'}}
		\| \{ 2^{k \alpha} \| \Delta_{k} g \| _{L^{p_2}} \}_{k \ge j} \|_{l^{\sigma}} \\
	&\to 0
\end{align*}
as $j\to \infty$.
Thus, we have
$\dot{T}_{V_j f} V_j g \to \dot{T}_f g$ in $\mathcal{S}'(\mathbb{R}^n)$ as $j\to \infty$.
In the same way, we can prove
$\dot{T}_{V_j g} V_j f \to \dot{T}_g f$
as $j \to \infty$.

Next, we write
\begin{align*}
	&\dot{R}(V_j f, V_j g) - \dot{R}(f, g) \\
	&=
	\dot{R}(V_j f_L -f_L, V_j g) + \dot{R}(V_j f_H - f_H, V_j g)
	+ \dot{R} (f, V_j g_L - g_L ) + \dot{R} (f, V_j g_H - g_H ).
\end{align*}
Taking the same $\mu_1 \in (q_1, \infty)$ as above, we compute
\begin{align*}
	\| \dot{R} (f, V_j g_L - g_L ) \|_{L^{\mu_1}}
	&=
	\left\| \sum_{\substack{k,l\in \mathbb{Z} \\  |k-l| \le 1}}
		\Delta_k f \Delta_l (V_j g_L - g_L) \right\|_{L^{\mu_1}} \\
	&=
	\left\| \sum_{l \le -j} \sum_{k=l-1}^{l+1} \Delta_k f \Delta_l (V_j g_L - g_L) \right\|_{L^{\mu_1}} \\
	&\lesssim
	\left( \sup_{k \le -j+1} \| \Delta_k f \|_{L^{\infty}} \right)
	\| \{ 2^{l \beta} \}_{l \le -j} \|_{l^1}
	\| \{ 2^{-l\beta} \| \Delta_l (V_j g_L - g_L) \|_{L^{\mu_1}} \}_{l\le -j} \|_{l^{\infty}} \\
	&\lesssim
	\left( \sup_{k \le -j+1} \| \Delta_k f \|_{L^{\infty}} \right)
	\| \{ 2^{l \beta} \}_{l \le -j} \|_{l^1}
	\| g \|_{\dot{B}^{-\beta}_{\mu_1,\infty}}
	\to 0
\end{align*}
as $j \to \infty$.
Moreover, we have
\begin{align*}
	\| \dot{R}  (f, V_j g_H - g_H ) \|_{L^{r}} 
	&=
	\left\| \sum_{\substack{k,l\in\mathbb{Z} \\ |k-l| \le 1}}
		\Delta_k f \Delta_l ( V_j g_H - g_H ) \right\|_{L^r} \\
	&=
	\left\| \sum_{l \ge j} \sum_{k=l-1}^{l+1} \Delta_k f \Delta_l ( V_j g_H - g_H ) \right\|_{L^r} \\
	&\lesssim
	\| f \|_{L^{q_2}}
	\| \{ 2^{-l \alpha} \}_{l \ge j} \|_{l^{\sigma'}}
	\| g \|_{\dot{B}^{\alpha}{p_2,\sigma}}
	\to 0
\end{align*}
as $j\to \infty$.
In the same way, we can see that
\begin{align*}
	\|  \dot{R}(V_j f_L -f_L, V_j g) \|_{L^{\mu_2}}
	&\to 0, \quad \text{where $\mu_2 \in (q_2, \infty)$},
	\\
	\|  \dot{R}(V_j f_H -f_H, V_j g) \|_{L^{r}}
	&\to 0
\end{align*}
as $j\to \infty$.
Thus, we conclude
$\dot{R}(V_j f, V_j g) - \dot{R}(f, g) \to 0$ in $\mathcal{S}'(\mathbb{R}^n)$
and \eqref{eq:app:b5} is proved.
\bigskip

\section{Proofs of Lemmas \ref{lem:meas:psi} and \ref{lem:meas:N}} \label{app:C}
\begin{proof}[Proof of Lemma \ref{lem:meas:psi}]
Let $t \in (0,T)$ be fixed.
Since $\psi \in Y(T) \subset L^{\infty}(0,T; \dot{\mathcal{B}}^{0}_{\sigma_1,2}(\mathbb{R}^n))$,
the function
$\psi (\cdot) : [0,t] \ni \tau \mapsto \psi(\tau) \in \dot{\mathcal{B}}^{0}_{\sigma_1,2}(\mathbb{R}^n)$
is strongly measurable.
By the proof of Lemma \ref{lem:Duha}, we see that
\begin{align}
    \| \mathcal{D}(t-\tau) \psi (\tau) \|_{\dot{B}^0_{r,2}}
    \lesssim
    \| \psi(\tau) \|_{\dot{B}^{0}_{\sigma_1,2}}
    + \| \psi (\tau) \|_{\dot{B}^0_{\mu,2}}
\end{align}
with
$\mu = \max \left\{\sigma_{1}, \left(\frac{1}{r}+\frac{1-\beta}{n}\right)^{-1}\right\}$,
where the implicit constant may
depend on $t$.
Since $\mu \in [\sigma_1, \sigma_2]$,
we further obtain
\begin{align}
    \| \mathcal{D}(t-\tau) \psi (\tau) \|_{\dot{B}^0_{r,2}}
    \lesssim
    \| \psi(\tau) \|_{\dot{B}^{0}_{\sigma_1,2}}
    + \| \psi (\tau) \|_{\dot{B}^0_{\sigma_2,2}}.
\end{align}
Hence, 
$\mathcal{D}(t-\tau) : 
\dot{\mathcal{B}}^0_{\sigma_1,2}(\mathbb{R}^n) \cap \dot{\mathcal{B}}^0_{\sigma_2,2}(\mathbb{R}^n)
\to \dot{\mathcal{B}}^0_{r,2}(\mathbb{R}^n)$
is a continuous map.

Let $\{ \psi_j(\cdot) \}_{j=1}^{\infty}$
be a sequence of step functions satisfying
$\lim_{j\to \infty} \psi_j = \psi$
in
$\dot{\mathcal{B}}^0_{\sigma_1,2}(\mathbb{R}^n) \cap \dot{\mathcal{B}}_{\sigma_2,2}(\mathbb{R}^n)$.
Since
$\mathcal{D}(t-\tau) \psi_j(\tau)$
is continuous in
$\dot{\mathcal{B}}^0_{r,2}(\mathbb{R}^n)$
for $\tau \in [0,t]$
except for finite points,
it is also strongly measurable.
Moreover, the above estimate implies
\[
    \lim_{j\to \infty} \mathcal{D}(t-\tau) \psi_j(\tau)
    = \mathcal{D}(t-\tau) \psi(\tau)
    \quad \text{in} \quad
    \dot{\mathcal{B}}^0_{r,2}(\mathbb{R}^n).
\]
Since the limit of strongly measurable function is strongly measurable,
the function
$\mathcal{D}(t-\cdot) \psi(\cdot) : [0,t] \to \dot{\mathcal{B}}^0_{r,2}(\mathbb{R}^n)$
is strongly measurable.
\end{proof}

\begin{proof}[Proof of Lemma \ref{lem:meas:N}]
First, we show the continuity of the map
\[
	\mathcal{N} :
	\dot{\mathcal{B}}^s_{2,2}(\mathbb{R}^n)\cap \dot{\mathcal{B}}^0_{r,2}(\mathbb{R}^n)
	\ni u \mapsto u^p \in  L^{\gamma}(\mathbb{R}^n)
\]
with $\gamma$ defined by
\[
    \frac{1}{\gamma} = \frac{1}{q} + \frac{1}{r},\quad
    \frac{1}{q} = \begin{dcases}
    \omega &(2s \ge n),\\
    \frac{(p-1)(n-2s)}{2n}
    &(2s<n),
    \end{dcases}
\]
where $\omega > 0$ is sufficiently small so that
$q \ge 2$ and $r \le q(p-1)$
hold.
Indeed, the H\"{o}lder inequality implies
\[
	\| \mathcal{N}(u) \|_{L^{\gamma}}
	\lesssim \| u \|_{L^r} \| u \|_{L^{q(p-1)}}^{p-1}
	\lesssim \| u \|_{\dot{B}^0_{r,2}}
    ( \| u \|_{\dot{B}^0_{r,2}}
    + \| u \|_{\dot{B}^s_{2,2}} )^{p-1},
\]
where we have used the embeddings
$\dot{\mathcal{B}}^0_{r,2}(\mathbb{R}^n) \subset L^{r}(\mathbb{R}^n)$
and
$\dot{\mathcal{B}}^0_{q(p-1),2}(\mathbb{R}^n) \subset L^{q(p-1)}(\mathbb{R}^n)$,
and
$\dot{\mathcal{B}}^s_{2,2}(\mathbb{R}^n)\cap \dot{\mathcal{B}}^0_{r,2}(\mathbb{R}^n) \subset \dot{\mathcal{B}}^0_{q(p-1),2}(\mathbb{R}^n)$,
which follows from the Gagliardo--Nirenberg inequality.
In particular, we have $\mathcal{N}(u) \in L^{\gamma}(\mathbb{R}^n) \subset \mathcal{S}'_h(\mathbb{R}^n)$.
Moreover, in the same way, we also deduce
\[
    \| \mathcal{N}(u) - \mathcal{N}(v) \|_{L^{\gamma}}
    \lesssim
    \| u - v \|_{\dot{B}^0_{r,2}}
    ( \| u \|_{\dot{B}^0_{r,2}}
    + \| u \|_{\dot{B}^s_{2,2}}
    + \| v \|_{\dot{B}^0_{r,2}}
    + \| v \|_{\dot{B}^s_{2,2}}
    )^{p-1},
\]
which shows the continuity of $\mathcal{N}$.

Let $u \in X(T)$.
Then, $u$ is strongly measurable from
$[0,T]$
to
$\dot{\mathcal{B}}^s_{2,2}(\mathbb{R}^n)\cap \dot{\mathcal{B}}^0_{r,2}(\mathbb{R}^n)$.
Combining this and the above continuity of $\mathcal{N}$,
we see that
$\mathcal{N}(u (\cdot)) : [0,T] \to L^{\gamma}(\mathbb{R}^n)$
is strongly measurable.
Next, we claim that for any
$\psi \in \dot{\mathcal{B}}^0_{\sigma_1',2}(\mathbb{R}^n)$,
\begin{equation}\label{eq:app:C:psi}
    \langle \mathcal{N}(u(\cdot)), \psi \rangle_{\dot{\mathcal{B}}^0_{\sigma_1,2}, \dot{\mathcal{B}}^0_{\sigma_1',2}} :
    [0,T] \to \mathbb{R}
\end{equation}
is measurable, where $\sigma_1'$ is the H\"{o}lder conjugate of $\sigma_1$.
Namely, the function
$\mathcal{N}(u(\cdot))$
is weakly measurable from
$[0,T]$
to
$\dot{\mathcal{B}}^0_{\sigma_1,2}(\mathbb{R}^n)$.
Since $\mathcal{S}_{\ast}(\mathbb{R}^n)$ is dense in
$\dot{\mathcal{B}}^0_{\sigma_1',2}(\mathbb{R}^n)$,
there exists a sequence
$\{ \phi_j \}_{j=1}^{\infty}$
of $\mathcal{S}_{\ast}(\mathbb{R}^n)$
such that
$\lim_{j\to \infty} \phi_j = \psi$
in 
$\dot{\mathcal{B}}^0_{\sigma_1',2}(\mathbb{R}^n)$.
For each $\phi_j$, the above
measurability of
$\mathcal{N}(u (\cdot)) : [0,T] \to L^{\gamma}(\mathbb{R}^n)$
implies that
the function 
\[
    \langle \mathcal{N}(u(\cdot)), \phi_j \rangle_{L^{\gamma}, L^{\gamma'}} :
    [0,T] \to \mathbb{R}
\]
is measurable.
Now, recalling the first part of the proof of Lemma \ref{lem:nonlin},
we know that
$\mathcal{N}(u(t)) \in \dot{\mathcal{B}}^0_{\sigma_1,2}(\mathbb{R}^n)$
for almost every $t \in [0,T]$.
This implies that
\[
    \langle \mathcal{N}(u(\cdot)), \phi_j\rangle_{\dot{\mathcal{B}}^0_{\sigma_1,2}, \dot{\mathcal{B}}^0_{\sigma_1',2}} :
    [0,T] \to \mathbb{R}
\]
is measurable.
Finally, taking the limit $j \to \infty$,
we conclude that the function \eqref{eq:app:C:psi} is measurable.

Moreover, we remark that 
$\dot{\mathcal{B}}^0_{\sigma_1,2}(\mathbb{R}^n)$
is separable,
since
$\mathcal{S}(\mathbb{R}^n)$
is a separable metric space
(see \cite[p. 144]{ReSiI}),
$\mathcal{S}_{\ast}(\mathbb{R}^n)$
is its subspace,
and $\mathcal{S}_{\ast}(\mathbb{R}^n)$
is dense in
$\dot{\mathcal{B}}^0_{\sigma_1,2}(\mathbb{R}^n)$.
Therefore, by the Pettis theorem
\cite[Theorem 1.1.6]{HyNiVeWe},
we conclude that 
the function
$\mathcal{N}(u(\cdot))$
is strongly measurable from
$[0,T]$
to
$\dot{\mathcal{B}}^0_{\sigma_1,2}(\mathbb{R}^n)$.
This completes the proof.
\end{proof}

\section{Proof of Lemma \ref{lem:XTM}} \label{app:E}
Since $X(T,M)$ is a closed subset of $X(T)$, it suffices to show that
$X(T)$ is complete.
Let $\{ \phi_n \}_{n \in \mathbb{N}} \subset X(T)$ be a Cauchy sequence.
Then, it is also the Cauchy sequence in
$L^{\infty} ( 0, T; \dot{\mathcal{B}}^s_{2,2}(\mathbb{R}^n) \cap \dot{\mathcal{B}}^0_{r,2}(\mathbb{R}^n) )$,
which is complete as stated in the introduction.
Thus, there exists a limit function
$\phi \in L^{\infty} ( 0, T; \dot{\mathcal{B}}^s_{2,2}(\mathbb{R}^n) \cap \dot{\mathcal{B}}^0_{r,2}(\mathbb{R}^n) )$.

It remains to prove $\phi \in X(T)$.
Let $M > 0$ be a constant such that
$\| \phi_n \|_{X(T)} \le M$ for all $n$.
Then, there exist zero sets $A_n \subset [0,T)$ for $n \in \mathbb{N}$ satisfying
\[
	\langle t \rangle^{\frac{s}{2}-\frac{n}{2}\left( \frac{1}{2} - \frac{1}{r} \right)} \| \phi_n(t) \|_{\dot{B}^s_{2,2}}
	+ \| \phi_n(t) \|_{\dot{B}^0_{r,2}} \le M
\]
for $t \in [0,T) \setminus A_n$.
Let $A = \bigcup_{n=1}^{\infty} A_n$, which is still a zero set.
Then, we have, for all $t \in [0,T) \setminus A$,
\begin{align}
	&\langle t \rangle^{\frac{s}{2}-\frac{n}{2}\left( \frac{1}{2} - \frac{1}{r} \right)} \| \phi (t) \|_{\dot{B}^s_{2,2}}
	+ \| \phi (t) \|_{\dot{B}^0_{r,2}} \\
	&=
	\lim_{n\to \infty}
	\left\{ \langle t \rangle^{\frac{s}{2}-\frac{n}{2}\left( \frac{1}{2} - \frac{1}{r} \right)} \| \phi_n(t) \|_{\dot{B}^s_{2,2}}
	+ \| \phi_n(t) \|_{\dot{B}^0_{r,2}} \right\} \\
	&\le \limsup_{n\to \infty} \| \phi_n \|_{X(T)} \\
	&\le M.
\end{align}
This implies $\| \phi \|_{X(T)} \le M$ and completes the proof.


\section*{Acknowledgements}
The authors are grateful to
Professor Akitaka Matsumura for 
suggesting the problem.
The authors also extend their sincere gratitude
to Professor Mamoru Okamoto for the continuous discussions and for providing the fundamental ideas
for the proofs of several lemmas in this paper.
This work was supported by JSPS KAKENHI Grant Numbers
JP19K14581,
JP20K14346,
JP22H00097,
JP24K06811,
JP24K16947.

\end{document}